\documentclass{amsart}
\usepackage{tikz}
\usepackage{amsfonts,amsmath,amsthm,amscd,amssymb,latexsym,amsbsy,pb-diagram,cite,caption,url}
\newtheorem{theorem}{Theorem}[section]
\newtheorem{corollary}[theorem]{Corollary}
\newtheorem{lemma}[theorem]{Lemma}
\newtheorem{proposition}[theorem]{Proposition}

\newtheorem{example}[theorem]{Example}
\theoremstyle{definition}
\newtheorem{definition}[theorem]{Definition}
\theoremstyle{remark}
\newtheorem{remark}[theorem]{Remark}
\numberwithin{equation}{section}

\DeclareMathOperator{\diam}{diam}
\newcommand{\mfu}{\mathfrak{U}}
\newcommand{\Sp}[1]{\operatorname{Sp}(#1)}
\newcommand{\abs}[1]{\left\vert#1\right\vert}
\newcommand{\B}{\mathbf{B}}
\numberwithin{equation}{section}

\begin{document}

\title[Extremal properties of finite ultrametric spaces]{Extremal properties and morphisms of finite ultrametric spaces and their representing trees}
\author{Oleksiy Dovgoshey}
\address{Institute of Applied Mathematics and Mechanics of NASU, Dobrovolskogo str. 1, Slovyansk 84100, Ukraine}
\email{oleksiy.dovgoshey@gmail.com}

\author{Evgeniy Petrov}
\address{Institute of Applied Mathematics and Mechanics of NASU, Dobrovolskogo str. 1, Slovyansk 84100, Ukraine}
\email{eugeniy.petrov@gmail.com}

\author{Hanns-Martin Teichert}
\address{Institute of Mathematics, University of L\"ubeck, Ratzeburger Allee 160, 23562 L\"ubeck, Germany}
\email{teichert@math.uni-luebeck.de}

\begin{abstract}
We study extremal properties of finite ultrametric spaces $X$ and related properties of representing trees $T_X$.   The notion of weak similarity for such spaces is introduced and related morphisms of labeled rooted trees are found.
It is shown that the finite rooted trees are isomorphic to the rooted trees of nonsingular balls of special finite ultrametric spaces. We also found conditions under which the isomorphism of representing trees $T_X$ and $T_Y$ implies the isometricity of ultrametric spaces $X$ and $Y$.
\end{abstract}

\keywords{Finite ultrametric space, sharp inequality for finite ultrametric space, isometry, ball-preserving mapping, isomorphism of rooted trees}        % the keywords

\subjclass[2010]{54E35, 05C05}      % MR(2000) Subject Classification

\maketitle
\tableofcontents
\section{Introduction}
In 2001 at the Workshop on General Algebra the attention of experts on the theory of lattices was paid to the following problem of I.~M.~Gelfand: Using graph theory describe up to isometry all finite ultrametric spaces~\cite{WGA}. An appropriate representation of ultrametric spaces $X$ by monotone rooted trees $T_X$ was proposed in ~\cite{GV}. The representation from~\cite{GV} can be considered in some sense as a solution of above mentioned problem. The question naturally arises about applications of this representation. One such application is the structural characteristic of finite ultrametric spaces for which the Gomory-Hu inequality becomes an equality, see~\cite{PD(UMB)}. The ultrametric spaces for which its representing trees are strictly binary were described in~\cite{DPT}. Our paper is also a contribution to this line of studies.

The first section of the paper contains the main definitions and the required technical results. In the second section we describe extremal and structural properties of finite ultrametric spaces which have the strictly $n$-ary representing trees and  the properties of such spaces having the representing trees with injective internal labeling. The main results here are Theorem~\ref{t1}, Theorem~\ref{t27} and Corollary~\ref{c2.12}. It is clear that the representing tree and the range of distance function are invariant under isometries of finite ultrametric spaces. Theorem~\ref{t3.5} and Theorem~\ref{t3.7} of the fifth section of the paper contain a description of spaces  for which the converse also holds: If finite ultrametric spaces $X$ and $Y$ have the same representing trees and the same range of distance functions, then $X$ and $Y$  are isometric.
The ball-preserving mappings, the isometries, and the weak similarities of finite ultrametric spaces are considered in the third section. The corresponding classes of morphisms of rooted trees are described in Theorem~\ref{t32}, Theorem~\ref{l3.3} and Theorem~{\ref{t4.3*}} respectively.
In the fourth section we prove that every finite rooted tree is isomorphic to the rooted tree of nonsingular balls of finite ultrametric space and find the minimal cardinality of this space in Theorem~\ref{t36}. The inclusions of finite rooted trees into balleans of some special ultrametric spaces are described  by Theorem~\ref{t3.6.2} and Theorem~\ref{t4.16}.

Recall some definitions from the theory of metric spaces and the graph theory.
\begin{definition}\label{d1.1}
	An \textit{ultrametric} on a set $X$ is a function $d\colon X\times X\rightarrow \mathbb{R}^+$, $\mathbb R^+ = [0,\infty)$, such that for all $x,y,z \in X$:
	\begin{itemize}
		\item [$(i)$] $d(x,y)=d(y,x)$,
		\item [$(ii)$] $(d(x,y)=0)\Leftrightarrow (x=y)$,
		\item [$(iii)$] $d(x,y)\leq \max \{d(x,z),d(z,y)\}$.
	\end{itemize}
\end{definition}
Inequality $(iii)$  is often called the {\it strong triangle inequality}. The \emph{spectrum} of an ultrametric space $(X,d)$ is the set
$$
	\operatorname{Sp}(X)=\{d(x,y)\colon x,y \in  X\}.
$$
The quantity
$$
\diam X=\sup\{d(x,y)\colon x,y\in X\}.
$$
is the diameter of $(X,d)$.
Recall that a \textit{graph} is a pair $(V,E)$ consisting of a nonempty set $V$ and a (probably empty) set $E$  whose elements are unordered pairs of different points from $V$. For a graph $G=(V,E)$, the sets $V=V(G)$ and $E=E(G)$ are called \textit{the set of vertices (or nodes)} and \textit{the set of edges}, respectively. If $\{x,y\} \in E(G)$, then the vertices $x$ and $y$ are \emph{adjacent}. A graph $G=(V,E)$ together with a function $w\colon E\rightarrow \mathbb{R}^+$ is called a \textit{weighted} graph, and  $w$ is called a \textit{weight} or a \textit{weighting function}. The weighted graphs will be denoted as $(G,w)$. A graph is \emph{complete} if $\{x,y\} \in E(G)$ for all distinct $x, y \in V(G)$. Recall that a \emph{path} is a nonempty graph $P=(V,E)$ whose vertices can be numbered so that
$$
V=\{x_0,x_1,...,x_k\}, \quad E=\{\{x_0,x_1\},...,\{x_{k-1},x_k\}\}.
$$
A finite graph $C$ is a \textit{cycle} if $|V(C)|\geq 3$ and there exists an enumeration $(v_1,v_2,...,v_n)$
 of its vertices such that
\begin{equation*}
(\{v_i,v_j\}\in E(C))\Leftrightarrow (|i-j|=1\quad \mbox{or}\quad |i-j|=n-1).
\end{equation*}
A graph $H$ is a \emph{subgraph} of a graph $G$ if
$$
V(H) \subseteq V(G) \ \text{ and } \ E(H) \subseteq E(G).
$$
We write $H\subseteq G$ if $H$ is a subgraph of $G$. A cycle $C$ is a \emph{Hamilton cycle} in a graph $G$ if $C\subseteq G$ and $V(C) =V(G)$. Similarly a path $P\subseteq G$ is a \emph{Hamilton path} in $G$ if $V(P) = V(G)$.
A connected graph without cycles is called a \emph{tree}.
A tree $T$ may have a distinguished vertex $r$ called the \emph{root}; in this case $T$ is called a \emph{rooted tree} and we write $T=T(r)$. A \emph{labeled rooted tree} $T=T(r,l)$ is a rooted tree $T(r)$ with a labeling $l\colon V(T)\to L$ where $L$ is a set of labels. Generally we  follow terminology used in~\cite{BM}.

\begin{definition}\label{def3.1}
A nonempty graph $G$ is called \emph{complete $k$-partite} if its vertices can be divided into disjoint nonempty sets $X_1,...,X_k$ so that there are no edges joining the vertices of the same set $X_i$ and any two vertices from different $X_i,X_j$, $1\leqslant i,j \leqslant k$ are adjacent. In this case we write $G=G[X_1,...,X_k]$.
\end{definition}
We shall say that $G$ is a {\it complete multipartite graph} if there exists
 $k \geqslant 2$ such that $G$ is complete $k$-partite.

\begin{definition}[\cite{DDP(P-adic)}]\label{d2}
Let $(X,d)$ be a finite ultrametric space. Define the graph $G_X^d$ as follows $V(G_X^d)=X$ and
$$
(\{u,v\}\in E(G_X^d))\Leftrightarrow(d(u,v)=\diam X).
$$
We call $G_X^d$ the \emph{diametrical graph} of $X$.
\end{definition}

\begin{theorem}[\cite{DDP(P-adic)}]\label{t13}
Let $(X,d)$ be a finite ultrametric space, $|X|\geqslant 2$. Then $G_X^d$ is complete multipartite.
\end{theorem}

For every nonempty finite ultrametric space $(X, d)$ we can associate  a labeled rooted tree $T_X=T_X(r,l)$ with $r=X$ and $l\colon V(T)\to \Sp{X}$ by the following rule (see~\cite{PD(UMB)}).

If $X=\{x\}$ is a one-point set, then $T_X$ is the rooted tree consisting of one node $X$ with the label~$\diam X=0$. Note that for the rooted trees consisting only of one node, we consider that this node is the root as well as a leaf.

Let $|X|\geqslant 2$. According to Theorem~\ref{t13} we have $G^d_X = G^d_X[X_1,...,X_k]$, $k \geqslant 2$.
In this case the root of the tree $T_X$ is labeled by $\diam X$ and, moreover, $T_X$ has the nodes $X_1,...,X_k$ of the first level with the labels
\begin{equation}\label{e2.7}
l(X_i)= \diam X_i,
\end{equation}
$i = 1,...,k$. The nodes of the first level with the label $0$ are leaves, and those indicated by strictly positive labels are internal nodes of the tree $T_X$. If the first level has no internal nodes, then the tree $T_X$ is constructed. Otherwise, by repeating the above-described procedure with $X_1,...,X_k$ instead of $X$, we obtain the nodes of the second level, etc. Since $X$ is finite, all vertices on some level will be leaves, and the construction of $T_X$ is completed.

The above-constructed labeled rooted tree $T_X$ is called the \emph{representing tree} of the ultrametric space $(X, d)$.
It should be noted that every finite ultrametric space is isometric to the set of leaves of an \emph{equidistant} tree with standard shortest-path metric. (See Figure~\ref{f1} for an example of equidistant tree.)
The descriptions of finite ultrametric spaces by representing trees and, respectively, equidistant trees are dual in a certain sense,  but the study of this duality is not a goal of the paper. Se also~\cite{GV, GNS00,GurVyal(2012),H04,Le} for correspondence between trees and ultrametric spaces.

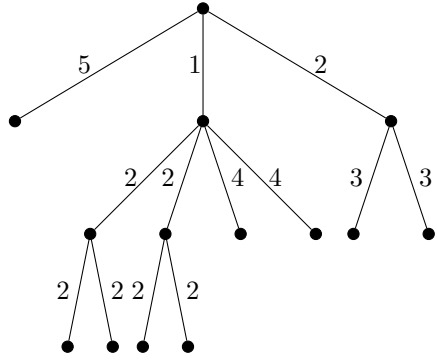
\begin{figure}
\begin{center}
\begin{tikzpicture}%[scale=1.5,font=\footnotesize]
% Specify spacing for each level of the tree
\tikzstyle{level 1}=[level distance=15mm,sibling distance=25mm]
\tikzstyle{level 2}=[level distance=15mm,sibling distance=10mm]
\tikzstyle{level 3}=[level distance=15mm,sibling distance=6mm]

\tikzset{
% Two node styles for game trees: solid and hollow
solid node/.style={circle,draw,inner sep=1.5,fill=black},
hollow node/.style={circle,draw,inner sep=1.5}
}

% The Tree
\node(0)[solid node]{}
    child{node(1)[solid node]{}edge from parent node[left]{$ \ 5 \ $}}
    child{node(1)[solid node]{}
        child{node[solid node]{}
        child{node[solid node]{} edge from parent node[left]{$2$}}
        child{node[solid node]{} edge from parent node[right]{$2$}}
        edge from parent node[left]{$2$}
        }
        child{node[solid node]{}
        child{node[solid node]{} edge from parent node[left]{$2$}}
        child{node[solid node]{} edge from parent node[right]{$2$}}
        edge from parent node[left]{$2$}
        }
        child{node[solid node]{} edge from parent node[right]{$4$}}
        child{node[solid node]{} edge from parent node[right]{$4$}}
        edge from parent node[left,xshift=3]{$1$}
    }
    child{node[solid node]{}
        child{node[solid node]{} edge from parent node[left]{$3$}}
        child{node[solid node]{} edge from parent node[right]{$3$}}
        edge from parent node[right,xshift=3]{$2$}
    };

\end{tikzpicture}
\end{center}
\caption{An equidistant tree.}
\label{f1}
\end{figure}

\begin{definition}\label{d14}
Let $(X,d)$ be an ultrametric space with $|X|\geqslant 2$ and  the spectrum $\operatorname{Sp}(X)$ and let $t\in \operatorname{Sp}(X)$ be nonzero. Define by $G_{t,X}$ a graph for which $V(G_{t,X})=X$ and
$$
(\{u,v\}\in E(G_{t,X}))\Leftrightarrow (d(u,v)=t).
$$
\end{definition}
 It is clear that $G_{t,X}$ is the diametrical graph of $X$ for $t=\operatorname{diam} X$.

\begin{definition}\label{d15}
Let $G=(V,E)$ be a nonempty graph, and let $V_0$ be the set (possibly empty) of all isolated vertices of $G$. Denote by $G'$ the subgraph of the graph $G$, induced by the set $V\backslash V_0$.
\end{definition}

Recall that a rooted tree is \emph{strictly $n$-ary} if every its internal node has exactly $n$ children. In the case $n=2$ such tree is called \emph{strictly binary}. The restriction of the labeling on the set of internal nodes of the labeled rooted tree~$T$ is called the \emph{internal labeling} of~$T$.

In 1961 E.\,C.~Gomory and T.\,C.~Hu \cite{GomoryHu(1961)}, for arbitrary finite ultrametric space $X$, proved the inequality
$$
	|\operatorname{Sp}(X)| \leqslant |X|.
$$
Denote by  $\mfu$ the class of finite ultrametric spaces $X$ with $\abs{\Sp{X}} = \abs{X}$.

\begin{theorem}[\cite{PD(UMB)}]\label{t1.7}
Let $(X, d)$ be a finite ultrametric space with $|X| \geqslant 2$. The following condidtions are equivalent.
\begin{itemize}
  \item [$(i)$] $(X, d) \in \mathfrak U$.
  \item [$(ii)$] $G'_{t,X}$ is complete bipartite for every  nonzero $t\in \operatorname{Sp}(X)$.
  \item [$(iii)$] $T_X$ is strictly binary and the internal labeling of $T_X$ is injective.
\end{itemize}
\end{theorem}

Another criterium of $X \in \mfu$ in terms of weighted Hamilton cycles and weighted Hamilton paths was proved in~\cite{DPT}.

Let $(X,d)$ be a finite ultrametric space and let $Y$ be a nonempty subspace of $X$. In the next proposition we identify $Y$ with the complete weighted graph $(G_Y,w)$ such that $V(G_Y)=Y$ and
\begin{equation}\label{eq2}
\forall \,  x,y\in Y \colon \quad	 w(\{x,y\})=d(x,y) \text{ if } x\neq y.
\end{equation}

\begin{proposition}[\cite{DPT}]\label{p10}
Let $(X,d)$ be a finite nonempty ultrametric space. The following conditions are equivalent.
\begin{itemize}
\item [$(i)$] $T_X$ is strictly binary.
\item [$(ii)$] If $Y\subseteq X$ and $|Y|\geqslant 3$, then there exists a Hamilton cycle in $(G_{Y},w)$ with exactly two edges of maximal weight.
\item [$(iii)$] There is no equilateral triangle in $(X,d)$.
\end{itemize}
\end{proposition}

In the next section of the paper we characterize the finite ultrametric spaces having the strictly $n$-ary representing trees and the finite ultrametric spaces having the representing trees with injective internal labeling.

\section{Injective internal labeling and strictly $n$-ary trees}\label{sc}

Let $T$ be a rooted tree with the root $r$. We  denote by $\overline{L}_T$ the set of the leaves of $T$, and, for every node $v$ of $T$,  by $T_v=T(v)$ the induced rooted subtree of $T$ with the root $v$ and such that
\begin{equation}\label{e2.1}
(u \in V(T_v)) \Leftrightarrow (u=v \text{ or $u$ is a successor of $v$ in $T$})
\end{equation}
holds for every $u\in V(T)$.
If $(X,d)$ is a finite ultrametric space and $T=T_X$, where $T_X$ is the representing tree of $X$, then for every node $v\in V(T)$ there are $x_1$, $\ldots$, $x_k \in X$ such that $\overline{L}_{T_v}=\{\{x_1\}, \ldots, \{x_k\}\}$. Thus $\overline{L}_{T_v}$ is a set of one-point subsets of $X$. In what follows we will use the notation $L_{T_v}$ for the set $\{x_1, \ldots, x_k\}$.

The following lemma was proved in~\cite{PD(UMB)} for the spaces $X\in \mathfrak U$ but its proof is also true for arbitrary finite ultrametric spaces.

\begin{lemma}\label{l2} Let $(X, d)$ be a finite ultrametric space with the representing tree $T_X$ and the labeling $l\colon V(T_X)\to \Sp{X}$ defined by~(\ref{e2.7}) and let $u_1 = \{x_1\}$ and $u_2 = \{x_2\}$ be two different leaves of the tree $T_X$. If
$(u_1, v_1, \ldots, v_n, u_2)$ is the path joining the leaves $u_1$ and $u_2$ in $T_X$, then
\begin{equation}\label{e1}
d(x_1, x_2) = \max\limits_{1\leqslant i \leqslant n} l({v}_i).
\end{equation}
\end{lemma}
Lemma~\ref{l2} implies in particular that the labeling $l\colon V(T_X)\to \Sp{X}$ is a surjective function for every nonempty finite ultrametric space $(X,d)$.

Recall that the \emph{union} $G_1 \cup G_2$ \emph{of graphs} $G_1$ and $G_2$ is the graph with the vertex set $V(G_1)\cup V(G_2)$ and the edge set $E(G_1)\cup E(G_2)$. If $V(G_1)\cap V(G_2)=\varnothing$, then the union $G_1\cup G_2$ is \emph{disjoint}.

\begin{lemma}\label{c25}
Let $(X,d)$ be a finite ultrametric space, let $T_X$ be its representing tree and let $t\in \Sp{X}\setminus \{0\}$. Then the graph $G'_{t,X}$ is the disjoint union of $p$ complete multipartite graphs $G^1_{t,X},...,G^p_{t,X}$ where $p$ is the number of nodes $x_1,...,x_p$ labeled by $t$.
Moreover for every $i\in \{1,...,p\}$ the graph $G^i_{r,X}$ is complete $k$-partite,
\begin{equation}\label{eq25}
G_{t,X}^i=G_{t,X}^i[L_{T_{s_{1}}},...,L_{T_{s_{k}}}],
\end{equation}
where $s_{1},...,s_{k}$ are the direct successors of $x_i$.
\end{lemma}
\begin{proof}
Suppose first that there is a single node $x_1$ labeled by $t$. Let $s_1,...,s_k$ be the direct successors of $x_1$. Consider the subtrees $T_{s_1},...,T_{s_k}$ of $T_X$ with roots  $s_1,...,s_k$.  Let
$$
x, y \in \bigcup\limits_{j=1}^k L_{T_{s_j}}.
$$
According to Lemma~\ref{l2} and to the construction of the representing trees the equality $d(x,y)=t$ holds if and only if $x$ and $y$ belong to the different sets $L_{T_{s_1}},....,L_{T_{s_k}}$. Using Definitions~\ref{d14} and~\ref{def3.1} we see that $G_{t,X}'$ is complete $k$-partite with parts $L_{T_{s_1}},....,L_{T_{s_k}}$.

Let $x_1$, $\ldots$, $x_p$ be the internal nodes labeled by $t$. According to the construction of the representing trees there are no $i, j \in \{1,..,p\}$ such that $x_i$ is a successor of $x_j$. This means that $L_{T_{x_i}}\cap L_{T_{x_j}} = \varnothing$ for all $i, j \in \{1,..,p\}$, $i\neq j$. Arguing as above we see that every node $x_i$ generates a complete multipartite graph $G_{t,X}^i$, with $V(G_{t,X}^i)=L_{T_{x_i}}$ such that $V(G_{t,X}^i)\cap V(G_{t,X}^j) = \varnothing$ for all distinct $i, j \in \{1,..,p\}$.
\end{proof}

Let $(X,d)$ be a metric space. Recall that a \emph{ball} with a radius $r \geqslant 0$ and a center $c\in X$ is the set $B_r(c)=\{x\in X\colon d(x,c)\leqslant r\}$. The \emph{ballean}  $\mathbf{B}_X$ of the metric space $(X,d)$ is the set of all balls of $(X,d)$. Every one-point subset of $X$ belongs to $\mathbf{B}_X$, we will say that this is a \emph{singular} ball in~$X$.

The following proposition claims that the ballean of finite ultrametric space $(X,d)$ is  the vertex set of representing tree $T_X$. The proof of the proposition can be found in~\cite{P(TIAMM)} but we reproduce it here for the convenience of the reader.
\begin{proposition}\label{lbpm}
 Let $(X,d)$ be a finite ultrametric space  with representing tree  $T_X$, $|X|\geqslant 2$. Then the following statements hold.
\begin{itemize}
  \item [$(i)$] $L_{T_v}\in \mathbf{B}_X$ holds for every node $v\in V(T_X)$.
  \item [$(ii)$] For every $B \in \mathbf{B}_X$ there exists the node $v$ such that $L_{T_v}=B$.
\end{itemize}
\end{proposition}
\begin{proof}
$(i)$ In the case where $v$ is a leaf of $T_X$ statement $(i)$ is evident.  Let $v$ be an internal node of $T_X$ and let $\{t\}\in L_{T_v}$. Consider the ball
$$
	B_{l(v)}(t)=\{x \in X\colon d(x,t)\leqslant l(v)\}.
$$
Let $\{t_1\} \in L_{T_v}$, $t_1 \neq t$. Since the path joining $\{t\}$ and $\{t_1\}$ lies in the tree $T_v$,  we have $d(t,t_1)\leqslant l(v)$ (see Lemma~\ref{l2}). The inclusion $L_{T_v}\subseteq {B}_{l(v)}(t)$ is proved. Conversely, suppose there exists $t_0\in B_{l(v)}(t)$ such that $\{t_0\} \notin L_{T_v}$. Consider the path $(\{t_0\},v_1,...,v_n,\{t\})$. It is clear that $\max\limits_{1\leqslant i \leqslant n}l(v_i)>l(v)$, i.e., $d(t_0,t)>l(v)$.  We have a contradiction.

$(ii)$ In the case $|B|=1$ statement $(ii)$ is evident. Let  $B\in \mathbf B_X$ such that $|B|\geqslant 2$. Let $x,y \in B$ with $d(x,y) = \diam B$. Consider the path $(\{x\},v_1,...,v_n, \{y\})$ in the tree $T_X$. According to Lemma~\ref{l2} we have $d(x,y) = \max\limits_{1\leqslant i \leqslant n}l(v_i)$. Let $i$ be an index such that $d(x,y)=l(v_i)$. To prove $(ii)$ it suffices to set $v=v_i$.
\end{proof}

\begin{remark}\label{r2.3*}
Let $(X,d)$ be a finite ultrametric space with $|X|\geqslant 2$. Let us define a function $\rho\colon V(T_X)\times V(T_X)\to \mathbb R^+$ as \begin{equation}\label{e2.6*}
\rho(u,v) = 0 \ \, \text{ if } \ \, u=v \ \, \text{ and } \ \,
\rho (u,v) =\max\limits_{1\leqslant i \leqslant n}l(v_i) \ \, \text{ if } \ \, u\neq v
\end{equation}
and $(v_1,...,v_n)$ is the path joining $v_1=u$ and $v_n = v$ in $T_X$. The function $\rho$ is an ultrametric on $V(T_X)$ (cf. Lemma~\ref{l2}). Hence, by Proposition~\ref{lbpm}, the function $\rho$ is also an ultrametric on the ballean $\mathbf B_X$. It can be shown that $\rho\colon \mathbf B_X \times \mathbf B_X \to \mathbb R^+$ is the restriction of \emph{Hausdorff distance} on the ballean $\mathbf B_X$ (see, for example, ~\cite{BBI} for the definition and properties of Hausdorff distance). Since $(\mathbf B_X, \rho)$ is a finite ultrametric space we can consider the representing trees $T_{\mathbf B_X}$, $T_{\mathbf B_{\mathbf B_X}}$, $T_{\mathbf B_{\mathbf B_{\mathbf B_X}}}$,... and so on, see Figure~\ref{f1*}. Some interesting properties of the sequence  $\mathbf B_X$, $\mathbf B_{\mathbf B_X}$, $\mathbf B_{\mathbf B_{\mathbf B_X}}$,..., have been  described by Derong Qiu~\cite{QD} but related them properties of representing trees remain unexplored.
\end{remark}

\begin{figure}
\begin{tikzpicture}[scale=0.8]
\foreach \i in {0}
{
 \foreach \j in {0}
 {
\draw (\i+2,\j)  node [right] {$T_X$};
\draw (\i+1,\j)  -- (\i+3,\j-2);
\draw (\i+2,\j-1)  -- (\i+1,\j-2);
\draw (\i+1,\j)  -- (\i,\j-1);
 \foreach \c in {(\i+1,\j),(\i+3,\j-2),(\i+2,\j-1),(\i+1,\j-2),(\i,\j-1)}
  \fill[black] \c circle (2pt);
}
}

\foreach \i in {5}
{
 \foreach \j in {0}
 {
\draw (\i+2,\j)  node [right] {$T_{\mathbf B_X}$};
\draw (\i+1,\j)  -- (\i+3,\j-2);
\draw (\i+2,\j-1)  -- (\i+1,\j-2);
\draw (\i+1,\j)  -- (\i,\j-1);
\draw (\i+2,\j-1)  -- (\i+2,\j-2);
 \foreach \c in {(\i+1,\j),(\i+3,\j-2),(\i+2,\j-1),(\i+1,\j-2),(\i,\j-1), (\i+2,\j-2)}
  \fill[black] \c circle (2pt);
}
}

\foreach \i in {10}
{
 \foreach \j in {0}
 {
\draw (\i+2,\j)  node [right] {$T_{\mathbf B_{\mathbf B_X}}$};
\draw (\i+1,\j)  -- (\i+3,\j-2);
\draw (\i+2,\j-1)  -- (\i+1,\j-2);
\draw (\i+1,\j)  -- (\i,\j-1);

\draw (\i+2,\j-1)  -- (\i+1+2/3,\j-2);
\draw (\i+2,\j-1)  -- (\i+1+4/3,\j-2);
 \foreach \c in {(\i+1,\j),(\i+3,\j-2),(\i+2,\j-1),(\i+1,\j-2),(\i,\j-1), (\i+1+2/3,\j-2), (\i+1+4/3,\j-2)}
  \fill[black] \c circle (2pt);
}
}

\end{tikzpicture}
\caption{The representing trees $T_X$, $T_{\mathbf B_X}$, and $T_{\mathbf B_{\mathbf B_X}}$ for the ultrametric space $(X,d)$ with $X=\{x_1,x_2,x_3\}$ and $d(x_1,x_2)=d(x_2,x_3)>d(x_1,x_3)$. }
\label{f1*}
\end{figure}
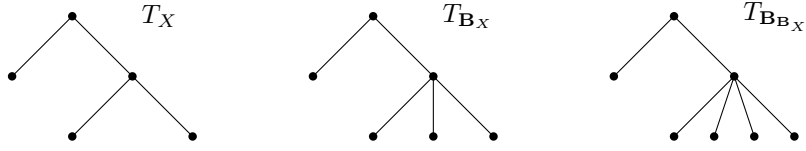

\begin{remark}\label{r2.4*}
The ultrametric $\rho$ defined by~(\ref{e2.6*}) can also be described as follows. Let us define a weight $w\colon E(T_X)\to \mathbb R^+$ such that, for all edges $\{u,v\}\in E(T_X)$,
$$
w(\{v,u\})=\max\{l(u), l(v)\},
$$
where $l(u)$ and $l(v)$ are defined by~(\ref{e2.7}). Then $\rho$ is the subdominant pseudoultrametric for the weight $w$ (see~\cite{DP2} and Theorem 10.40 in~\cite{SD} for details).
\end{remark}

\begin{theorem}\label{t1}
Let $(X,d)$ be a finite nonempty ultrametric space. The following conditions are equivalent.
\begin{itemize}
\item [$(i)$] The diameters of different nonsingular balls are different.
\item [$(ii)$] The internal labeling of $T_X$ is injective.
\item [$(iii)$] $G'_{t,X}$ is a complete multipartite graph for every $t\in \operatorname{Sp}(X)\setminus \{0\}$.
\item [$(iv)$] The equality
$$
	|\mathbf{B}_X|=|X|+|\Sp{X}|-1
$$
holds.
\end{itemize}
\end{theorem}
\begin{proof}
The theorem is trivial if $|X|=1$. Suppose that $|X|\geqslant 2$. Proposition~\ref{lbpm} and the definition of the representing trees imply that $(i)$ and $(ii)$ are equivalent. The implication $(ii)\Rightarrow(iii)$ follows from Lemma~\ref{c25}.

Let us prove $(iii)\Rightarrow (ii)$.  Suppose that $G'_{t,X}$ is a complete multipartite graph for every  $t\in \operatorname{Sp}(X)\setminus \{0\}$. Consider the case where there exist exactly two different internal nodes $u$ and $v$ of $T_X$ with $l(u)=l(v)$. According to properties of the representing trees $u$ and $v$ are not incident. Let $u_1,..,u_{m}$ and $v_1,..,v_{n}$ be the direct successors of $u$ and $v$ respectively. Arguing as in proof of Lemma~\ref{c25} we see that  $G'_{l,X}= U\cup W$, where $l=l(u)=l(v)$ and
$$
U=U[L_{T_{u_1}}, \ldots, L_{T_{u_{m}}}],\ \, \, W=W[L_{T_{v_1}},\ldots ,L_{T_{v_{n}}}]\ \text{ and } \  V(U)\cap V(W)=\varnothing.
$$
We have a contradiction since the disjoint union of two complete multipartite graphs is disconnected. The case where the number of different internal nodes having equal labels is more than two is analogous.

The equivalence $(ii)\Leftrightarrow(iv)$ follows from Proposition~\ref{lbpm} and from the equality
$$
|Y|=|f(Y)|,
$$
which holds for every injective mapping $f$ defined on a set $Y$. The proof is completed.
\end{proof}

\begin{remark}\label{r2.5}
The inequality
\begin{equation}\label{e56}
|\Sp{X}|\leqslant |\mathbf{B}_X|-|X|+1
\end{equation}
holds for every finite nonempty ultrametric space $X$.
Indeed, it follows from Proposition~\ref{lbpm} that
$$
	|\mathbf B_X|=|V(T_X)|.
$$
Each node of $T_X$ is either an internal node or a leaf. Since $|\Sp{X}|-1$ is no more than the number of internal nodes and since the number of leaves is $|X|$, inequality~(\ref{e56}) holds. Thus, by Theorem~\ref{t1}, the internal labeling of $T_X$ is injective if and only if inequality~(\ref{e56}) becomes an equality.
\end{remark}

\begin{lemma}\label{s}
Let $X$ be a finite nonempty ultrametric space. Suppose that there is a natural number $n\geqslant 2$ such that the equality
\begin{equation}\label{x0}
(n-1)|\mathbf B_Y |+1 = n |Y|
\end{equation}
holds for every ball $Y \in \mathbf B_X$. Then the following statements are equivalent for every nonsingular ball $Y \in \mathbf B_X$.
\begin{itemize}
  \item [$(i)$] The equality $|Y|=n$ holds.
  \item [$(ii)$] All children of the node $Y$ are leaves of $T_X$.
\end{itemize}
\end{lemma}
\begin{proof}
Let $|Y|=n$ hold. Then from~(\ref{x0}) it follows that
$$
(n-1)|\mathbf B_Y|=n^2-1.
$$
Since $n-1\neq 0$, the last equality implies that $|\mathbf B_Y|=n+1$. The set $Y$ and the one-point subsets of $Y$ are elements of $\mathbf B_Y$. Hence, the equalities $|\mathbf B_Y|= n+1$ and $|Y|=n$ give us either $B=Y$ or $|B|=1$ for every $B \in \mathbf B_Y$. Statement $(ii)$ follows from $(i)$.

The converse also holds. Indeed, suppose that $Y\in \mathbf B_X$ is nonsingular and all children of $Y$ are leaves. Then we have the equality
$$
|\mathbf B_Y|=|Y|+1.
$$
This equality and~(\ref{x0}) imply the equation
$$
(n-1)(|Y|+1)+1=n|Y|
$$
which has the unique solution $|Y|=n$.
\end{proof}

\begin{remark}\label{r2.7}
	If $|Y|=1$, then we have also $|\mathbf{B}_Y|=1$. Consequently, equality~\eqref{x0} holds for every ultrametric space $X$ and every positive integer number $n$ if~$Y$ a singular ball.
\end{remark}

\begin{figure}
\begin{tikzpicture}[yscale=0.75]
\draw (0,3) -- (1,6) -- (4,8) -- (7,6) -- (8,3);
 \foreach \i in {(4,8),(1,6),(4,6),(7,6),(0,3),(1,3),(2,3),(3,3),(4,3),(5,3),(6,3),(7,3),(8,3),(1,0),(2,0),(3,0),(4,0),(5,0),(6,0)}
  \fill[black] \i+(2pt,-1pt) arc(0:360:2pt and 2.66pt); %circle (3pt);
%  \fill[black] (0,0)  arc(0:360:3pt and 6pt);
\draw (1,3) -- (1,6) -- (2,3) -- (3,0);
\draw (3,3) -- (4,6) -- (5,3) -- (6,0);
\draw (6,3) -- (7,6) -- (8,3);
\draw (1,6) -- (1,3);
\draw (1,0) -- (2,3) -- (2,0);
\draw (4,6) -- (4,3);
\draw (4,0) -- (5,3) -- (5,0);
\draw (6,3) -- (7,6) -- (7,3);
\draw (4,8) -- (4,6);
\end{tikzpicture}
\caption{Example of strictly $3$-ary tree.}
\label{fig1}
\end{figure}
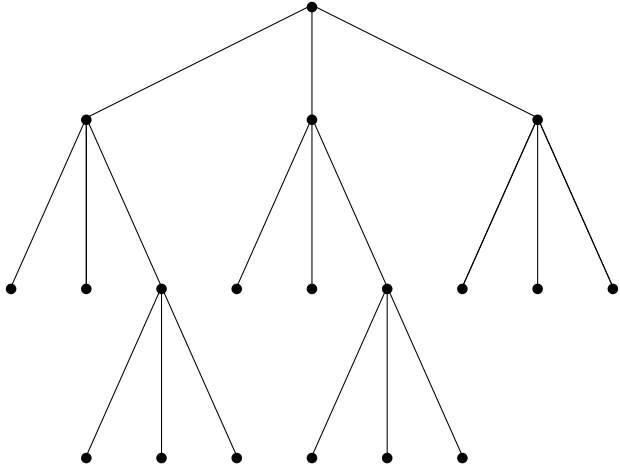

Let $(X,d)$ be a metric space. Recall that balls $B_1$, $\ldots$, $B_k$ in $(X,d)$ are \emph{equidistant} if there is $r>0$ such that $d(x_i, x_j)=r$ holds whenever $x_i \in B_i$ and $x_j \in B_j$ and $1\leqslant i< j \leqslant k$. Every two disjoint balls $B_1$, $B_2$  are equidistant in any ultrametric space.

\begin{theorem}\label{t27}
Let $(X,d)$ be a finite nonempty ultrametric space and let $n\geqslant 2$ be integer. The following conditions are equivalent.
\begin{itemize}
\item [$(i)$] $T_X$ is strictly $n$-ary.
\item [$(ii)$] For every nonzero $t\in \Sp{X}$, the graph $G'_{t,X}$ is the union of $p$ complete $n$-partite graphs, where $p$ is a number of all internal nodes of $T_X$ labeled by $t$.
\item [$(iii)$] For every nonsingular ball $B \in \mathbf B_X$, there are equidistant disjoint balls $B_1,...,B_n \in \mathbf B_X$ such that $B=\bigcup\limits_{j=1}^n B_j$.
\item [$(iv)$] Equality~(\ref{x0}) holds for every ball $Y\in \mathbf B_X$.
\end{itemize}
\end{theorem}
\begin{proof}
$(i)\Rightarrow(ii)$. This implication follows directly from Lemma~\ref{c25}.

$(ii)\Rightarrow(i)$. Let condition $(ii)$ hold and let $x_1,..,x_p$ be the internal nodes of~$T_X$ labeled by $t\in \Sp{X}\setminus \{0\}$. Consider the trees $T_{x_i}$ with the roots $x_i$, $i=1, \ldots, p$. It follows from the definition of the representing trees that the sets $L_{T_{x_1}},...,L_{T_{x_p}}$ are disjoint. Hence every complete $n$-partite graph from the union mentioned in condition~$(ii)$ is generated by an internal node $x_i$ labeled by $t$. Using~$(ii)$ we see that the node $x_i$ has exactly $n$ direct successors  for every $i\in \{1,...,p\}$.

The construction of the representing trees and Proposition~\ref{lbpm} yield the equivalence $(i)\Leftrightarrow (iii)$.

$(i)\Rightarrow(iv)$. Suppose $(i)$ holds. Let us denote by $N_{Y}$ the number of the internal nodes of $T_Y$ for arbitrary nonsingular $Y\in \mathbf B_Y$. It follows directly from the definition of the strictly  $n$-ary rooted trees and Proposition~\ref{lbpm} that $T_Y$ is strictly $n$-ary. Hence we have
$nN_{Y}+1=|V(T_Y)|$. Proposition~\ref{lbpm} implies
$|V(T_Y)|=|\mathbf B_Y|$ and $N_{Y}=|\mathbf B_Y|-|Y|$. Consequently,
$$
n(|\mathbf B_Y|-|Y|)+1 = |\mathbf B_Y|
$$
holds. The last equality and~(\ref{x0}) are equivalent.

$(iv)\Rightarrow(i)$. Let $(iv)$ hold. Equality~(\ref{x0}) implies that
\begin{equation}\label{x1}
|Y|=|\mathbf B_Y|+\frac{1}{n}-\frac{|\mathbf B_Y|}{n}
\end{equation}
for every nonsingular ball $Y\in \mathbf B_X$.
Since
\begin{equation}\label{x2}
|\mathbf B_Y|\geqslant |Y|+1,
\end{equation}
from ~(\ref{x1}) we obtain
$$
|\mathbf B_Y|\geqslant n+1.
$$
Now using~(\ref{x1}) we see that
$$
	|Y|=\left(1-\frac{1}{n}\right)|\mathbf B_Y|+\frac{1}{n} \geqslant \frac{(n-1)(n+1)}{n}+\frac{1}{n}=n.
$$
Thus every nonsingular ball $Y\in \mathbf B_X$ contains at least $n$ distinct points.

Now we can prove $(i)$ by induction on $|X|$.

It was proved that $|Y|\geqslant n$ holds for every nonsingular ball $Y\in\mathbf B_X$. Consequently we have $|X|\geqslant n$. If $|X|=n$, then statement $(i)$ follows from Lemma~\ref{s}. Suppose $(i)$ does not hold if $|X|=m$ but $(i)$ holds if $n \leqslant |X|<m$. The space $(X,d)$ contains a nonsingular ball $Y\in \mathbf B_Y$ such that all successors of the node $Y$ are leaves of $T_X$. Define a set~$X^*$ as
$$
X^*=(X\setminus Y)\cup \{y^*\},
$$
where $y^*$ is an arbitrary point of $Y$. Since $|Y|\geqslant 2$, we have the inequality
$$
|X^*|<|X|.
$$
To complete the proof of $(i)$ it suffices, by the induction hypothesis, to show that
\begin{equation}\label{s2}
(n-1)|\mathbf B_W|+1=n|W|
\end{equation}
holds for every $W\in \mathbf B_{X^*}$. Let $W\in \mathbf B_{X^*}$. Equality~(\ref{s2}) is trivial if $y^*\notin W$. Let $y^* \in W$. Then the set
$$
\widetilde{W}=W\cup Y
$$
is a nonsingular ball in $(X,d)$. It is easy to prove that
$$
|\widetilde{W}|=|W|+|Y|-1 \quad \text{and} \quad |\mathbf B_W|=|\mathbf B_{\widetilde W}|+|Y|.
$$
These  equalities and Lemma~\ref{s} give us
\begin{equation}\label{x4}
    |\widetilde{W}|=|W|+n-1 \ \text{ and } \ |\mathbf B_{\widetilde W}|=|\mathbf B_W|+n.
\end{equation}
Using~(\ref{x4}) and equality~(\ref{x0}) with $Y=\widetilde{W}$ we obtain
$$
(n-1)(|\mathbf B_W|+n)+1= n(|W|+n-1).
$$
Equality~(\ref{s2}) follows.
\end{proof}

Let $T$ be a rooted tree and let $v$ be a node of $T$. Denote by $\delta^+(v)$ the \emph{out-degree} of $v$, i.e., $\delta^+(v)$ is the number of children of $v$, and write
$$
	\Delta^+(T) = \max_{v \in V(T)} \delta^+(v),
$$
i.e., $\Delta^+(T)$ is the maximum out-degree of $V(T)$. It is clear that $v \in V(T)$ is a leaf of $T$ if and only if $\delta^+(v)=0$. Moreover, $T$ is strictly $n$-ary if and only if the equality
$$
	\delta^+(v)=n
$$
holds for every internal node $v$ of $T$. Let us denote by $I(T)$ the set of all internal nodes of $T$.

\begin{lemma}\label{l2.9}
Let $T$ be a finite rooted tree.	Then the inequality \begin{equation}\label{l2.9e1}
	|V(T)| \leqslant \Delta^+(T) |I(T)|+1
	\end{equation}
	holds. If $|V(T)| \geqslant 2$, then this inequality becomes the equality if and only if $T$ is strictly $n$-ary with $n=\Delta^+(T)$.
\end{lemma}
\begin{proof}
 It is clear that \begin{equation}\label{l2.9e2}
	|E(T)| = \sum_{v \in I(T)} \delta^+(v).
	\end{equation}
	Since $|V(T)|=|E(T)|+1$ holds (see, for example, \cite[Corollary 1.5.3]{Di}) and we have
	\begin{equation}\label{l2.9e3}
	\sum_{v \in I(T)} \delta^+(v) \leqslant \Delta^+(T) |I(T)|,
	\end{equation}
	inequality~\eqref{l2.9e1} follows.
	
	It easy to see that inequalities~\eqref{l2.9e1} and~\eqref{l2.9e3} become equalities simultaneously. Since $\delta^+(v) \leqslant \Delta^+(T)$ holds for every $v \in V(T)$, inequality~\eqref{l2.9e3} becomes an equality if and only if we have
	$$
		\delta^+(v) = \Delta^+(T)	
	$$
	for every $v \in I(T)$. The last condition means that $T$ is strictly $n$-ary with $n = \Delta^+(T)$.
\end{proof}

\begin{corollary}\label{c2.10}
	The inequality
	\begin{equation}\label{c2.10e1}
		|\mathbf{B}_X| \geqslant \frac{\Delta^+(T_X)|X|-1}{\Delta^+(T_X)-1}
	\end{equation}
	holds for every finite nonempty ultrametric space $(X,d)$. This inequality becomes an equality if and only if $T_X$ is a strictly $n$-ary rooted tree with $n = \Delta^+(T_X)$.
\end{corollary}
\begin{proof}
	If $|X|=1$, then we have $|\mathbf{B}_X|=1$ and $\Delta^+(T_X)=0$. Thus
	$$
	\frac{\Delta^+(T_X)|X|-1}{\Delta^+(T_X)-1}=\frac{0-1}{0-1}=1=|\mathbf{B}_X|.
	$$
	Suppose $|X| \geqslant 2$ holds. It follows from Proposition~\ref{lbpm} that
	$$
		|I(T_X)|=|\mathbf{B}_X|-|X| \quad \text{and} \quad |V(T_X)| = |\mathbf{B}_X|.
	$$
	Thus~\eqref{c2.10e1} is an equivalent form of~\eqref{l2.9e1} for $T=T_X$.
\end{proof}

\begin{remark}\label{r2.11}
	We have $\Delta^+(T_X)-1 \neq 0$ in inequality~\eqref{p2.11e1} (see Lemma~\ref{l3.4} of the present paper).
\end{remark}

Using Corollary~\ref{c2.10} and Remark~\ref{r2.5} we obtain
\begin{proposition}\label{p2.11}
	Let $(X, d)$ be a finite ultrametric space with $|X|\geq 2$. Then the inequality
	\begin{equation}\label{p2.11e1}
	2|\mathbf{B}_X| \geqslant |\operatorname{Sp}(X)| + \frac{2\Delta^+(T_X)|X|-\Delta^+(T_X)-|X|}{\Delta^+(T_X) - 1}
	\end{equation}
	holds. This inequality becomes an equality if and only if $T_X$ is a strictly $n$-ary rooted tree with injective internal labeling and $n=\Delta^+(T_X)$.
\end{proposition}
\begin{proof}
	Note that the right side of~\eqref{p2.11e1} is the sum of the right sides of inequalities~\eqref{e56} and~\eqref{c2.10e1}. The proposition follows from Remark~\ref{r2.5} and Corollary~\ref{c2.10}.
\end{proof}

\begin{corollary}\label{c2.12}
Let $(X,d)$ be a finite ultrametric space with $|X|\geqslant 2$ and let $n = \Delta^+(T_X)$. The following conditions are equivalent.
\begin{itemize}
  \item [$(i)$] $T_X$ is a strictly $n$-ary tree with injective internal labeling.
  \item [$(ii)$] $G'_{r,X}$ is a complete $n$-partite graph for every $r\in \Sp{X}$.
  \item [$(iii)$] The equality
  \begin{equation}\label{c2.12e1}
	  2|\mathbf{B}_X| = |\operatorname{Sp}(X)| + \frac{2\Delta^+(T_X)|X|-\Delta^+(T_X)-|X|}{\Delta^+(T_X) - 1}
  \end{equation}
  holds.
\end{itemize}
\end{corollary}

Theorem~\ref{t1.7} and Corollary~\ref{c2.12} show, in particular, that the equality $|X|=|\operatorname{Sp}(X)|$ holds if and only if we have~\eqref{c2.12e1} and $\Delta^+(T_X)=2$.

In the following proposition we suppose that the weighted graph $(G_Y, w)$ is defined as in Proposition~\ref{p10}.

\begin{proposition}\label{p2.13}
Let $(X,d)$ be a finite ultrametric space with $|X|\geqslant 3$ and let $n\geqslant 2$ be a natural number. If $T_X$ is strictly $n$-ary, then for every nonsingular ball~$Y\in  \B_X$, the graph $(G_Y,w)$ contains a Hamilton cycle with exactly $n$ edges of maximal weight.
\end{proposition}
\begin{proof}
Let $T_X$ be strictly $n$-ary and let $Y\in \B_X$ with $|Y|> 1$. According to Proposition~\ref{lbpm} there exists an internal node $x_0$ of $T_X$ such that  $Y=L_{T_{x_0}}$. Since $T_X$ is strictly $n$-ary, $x_0$ has $n$ direct successors $x_1,...,x_n$. Consider the subtrees $T_{x_1},...,T_{x_n}$ with the roots  $x_1,...,x_n$. By Lemma~\ref{l2}, for any $x,y\in Y$, the equality $d(x,y) = l(x_0)$ holds if $x$ and $y$ belong to different sets $L_{T_{x_i}}$, $L_{T_{x_j}}$, $1\leqslant i,j\leqslant n$. Similarly we have $d(x,y) < l(x_0)$ if $x$ and $y$ belong to the same set $L_{T_{x_i}}$. It is easy to see that the cycle
$$
C=(x_{11},...,x_{1k_1},x_{21},...,x_{2k_2},...,x_{n1},...,x_{nk_n})
$$
is a Hamilton cycle in $(G_Y,w)$ and has exactly $n$ edges of maximal weight $l(x_0)=\diam Y$ where $$\{x_{11},...,x_{1k_1}\}=L_{T_{x_1}},..., \{x_{n1},...,x_{nk_n}\}=L_{T_{x_n}}.$$
\end{proof}

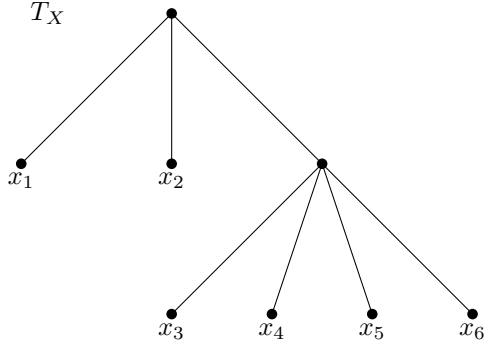
\begin{figure}
\begin{tikzpicture}[scale=1]
\draw (0,4) node [right] {$T_X$};
\draw (0,2) node [below] {$x_1$} -- (2,4) --  (4,2) -- (6,0) node [below] {$x_6$};
\fill[black] (0,2) circle(2pt); \fill[black] (2,4) circle(2pt); \fill[black] (4,2) circle(2pt); \fill[black] (6,0) circle(2pt);
\draw (2,4)  -- (2,2) node [below] {$x_2$};\fill[black] (2,2) circle(2pt);
\draw (4,2)  -- (2,0) node [below] {$x_3$};\fill[black] (2,0) circle(2pt);
\draw (4,2)  -- (10/3,0) node [below] {$x_4$};\fill[black] (10/3,0) circle(2pt);
\draw (4,2)  -- (14/3,0) node [below] {$x_5$};\fill[black] (14/3,0) circle(2pt);
\end{tikzpicture}
\caption{$T_X$ is not strictly $4$-ary.}
\label{fig2}
\end{figure}

\begin{remark}\label{r2.15}
The inversion of Proposition~\ref{p2.13} is false. The existence, for every nonsingular ball $Y \in \mathbf{B}_X$, of a Hamilton cycle with exactly $n$ edges of maximal weight in $(G_Y,w)$, does not guarantee that the rooted tree $T_X$ is strictly $n$-ary.
\end{remark}

\begin{example}\label{ex2.16}
	Let $(X,d)$ be an ultrametric space with $X=\{x_1,x_2,x_3,x_4,x_5,x_6\}$ and with $T_X$ depicted at Figure~\ref{fig2}. There are only two nonsingular balls $B_1=X$ and $B_2=\{x_3,x_4,x_5,x_6\}$ in $(X,d)$. It is clear that both Hamilton cycles $C_1=(x_1,x_3,x_2,x_4,x_5,x_6)$ and $C_2=(x_3,x_4,x_5,x_6)$ of the weighted graphs $(G_{B_1},w)$ and $(G_{B_2},w)$, respectively, have exactly four edges of maximal weight.
\end{example}

\section{Mappings of ultrametric spaces and morphisms of rooted trees}

In this section we describe some interrelations between mappings of finite ultrametric spaces and morphisms of their representing trees. First of all we recall the definition of isomorphic rooted trees.

\begin{definition}\label{d3.1}
Let $T_1=T_1(r_1)$  and $T_2=T_2(r_2)$ be rooted trees. A bijection $f\colon V(T_1)\to V(T_2)$ is an isomorphism of $T_1$ and $T_2$ if $f(r_1)=r_2$ and
\begin{equation}\label{d3.1e1}
(\{u,v\} \in E(T_1))\Leftrightarrow (\{f(u),f(v)\} \in E(T_2))
	\end{equation}
	holds for all $u$, $v \in V(T)$. The rooted trees $T_1$ and $T_2$ are isomorphic if there exists an isomorphism $f\colon V(T_1)\to V(T_2)$.
\end{definition}

In what follows we write $T_1(r_1) \simeq T_2(r_2)$ if the rooted trees $T_1(r_1)$ and $T_2(r_2)$ are isomorphic.

Let $X$ and $Y$ be metric spaces. The mapping $F\colon X\to Y$ is ball-preserving if for all $Z\in \textbf{B}_X$ and $W\in \textbf{B}_Y$, the relations
$$
F(Z)\in \textbf{B}_Y \,\text{ и } \, F^{-1}(W)\in \textbf{B}_X
$$  hold, where $F(Z)$ is the image of the set $Z$ under the mapping $F$ and $F^{-1}(W)$ is the  preimage of the set $W$ under this mapping.
\begin{theorem}\label{t32}
Let $X$ and $Y$ be finite ultrametric spaces. Then $T_X$ and $T_Y$ are isomorphic as rooted trees if and only if there exists a bijective ball-preserving mapping $f\colon X\to Y$.
\end{theorem}
\begin{proof}
Suppose first that $f\colon X\to Y$  is bijective and ball-preserving. Then the mapping
\begin{equation}\label{x1s}
\mathbf B_X \ni B\mapsto f(B) \in \mathbf B_Y
\end{equation}
is also bijective. According to Proposition~\ref{lbpm} there exist the bijections
$$
\Gamma_1\colon V(T_X) \to \mathbf B_X \, \ \text{  and } \ \,  \Gamma_2\colon V(T_Y) \to \mathbf B_Y,
$$
 such that
 $\Gamma_1(v)=L_{T_v}$ and $\Gamma_2(u)=L_{T_u}$ for all $v\in V(T_X)$ and  $u\in V(T_Y)$ (see~(\ref{e2.1})). Define the mapping $\Phi\colon V(T_X)\to V(T_Y)$ by
$$
\Phi(u)=\Gamma_2^{-1}(f(\Gamma_1(u))).
$$
Since $f$ is ball-preserving, $\Phi$ is well-defined. The bijectivity of $\Phi$ follows from the bijectivity of $\Gamma_1$, $\Gamma_2$ and of the mapping defined by~(\ref{x1s}).

For the root $X$ of $T_X$ we have $\Gamma_1(X)=X$, $f(X)=Y$ and $\Gamma_2^{-1}(Y)=Y$. Hence $\Phi(X)=\Gamma_2^{-1}f(\Gamma_1(X))=Y$, i.e. $\Phi(X)$ is the root of $T_Y$. Consequently to prove that $\Phi$ is an isomorphism of $T_X$ and $T_Y$ we have to establish  equivalence~(\ref{d3.1e1})
 with $T_1=T_X$, $T_2=T_Y$ for all $u,v \in V(T_X)$.
It is easy to see that the left part of~(\ref{d3.1e1}) holds if and only if the implication
\begin{equation}\label{e101}
 (\Gamma_1(u)\subseteq B \subseteq \Gamma_1(v)) \Rightarrow ((\Gamma_1(u)=B)\vee(\Gamma_1(v)=B))
\end{equation}
holds for every $B\in \mathbf B_X$. Similarly the right part of~(\ref{d3.1e1}) holds if and only if
\begin{equation}\label{e102}
(\Gamma_2(\Phi(u))\subseteq B \subseteq \Gamma_2(\Phi(v))) \Rightarrow ((\Gamma_2(\Phi(u))=B)\vee(\Gamma_2(\Phi(v))=B))
\end{equation}
for every $B \in \mathbf B_X$.
Using~(\ref{x1s}) we can rewrite~(\ref{e101}) as
\begin{equation}\label{e103}
(f(\Gamma_1(u))\subseteq f(B) \subseteq f(\Gamma_1(v))) \Rightarrow ((f(\Gamma_1(u))=f(B))\vee(f(\Gamma_1(v))=f(B))).
\end{equation}
Taking into consideration that $\Gamma_2(\Phi(u))=f(\Gamma_1(u))$ we see that ~(\ref{e103}) is equivalent to ~(\ref{e102}).

Conversely, suppose that $T_X\simeq T_Y$ with the isomorphism $F\colon V(T_X)\to V(T_Y)$. Denote by $B_X^1$ and $B_Y^1$  the sets of singular balls of the spaces $X$ and $Y$ respectively.
Since $F$ is an isomorphism,  the mapping
\begin{equation*}
f^*=F|_{B_X^1}\colon B_X^1\to B_Y^1
\end{equation*}
is a bijection. Define a mapping $f\colon X\to Y$ as follows
$$
(f(x)=y) \ \, \Leftrightarrow \ \, (f^*(\{x\})=\{y\}).
$$
We claim that $f$ is ball-preserving.

It is clear that for $B\in \mathbf B_X$ the image $f(B)$ is  a ball in $Y$ if and only if the equality
\begin{equation}\label{eq110}
f^*\left(\bigcup\limits_{x\in B}\{\{x\}\}\right) = \bigcup\limits_{y \in B'} \{\{y\}\}
\end{equation}
holds  for some $B'\in \mathbf B_Y$.

Let $B$ be a ball in $X$.
Let $b\in V(T_X)$ such that $\Gamma_1(b)=B$. Since $T_X \simeq T_Y$ and $F$ is an isomorphism we have
\begin{equation}\label{eq120}
T_b\simeq T_{F(b)},
\end{equation}
where $T_b$ and $T_{F(b)}$ are the corresponding subtrees of $T_X$ and $T_Y$, respectively (see~(\ref{e2.1})).

Let
$$
\widetilde{\Gamma}_1(b)=\bigcup\limits_{x\in B}\{\{x\}\} \ \text{ and  } \ \widetilde{\Gamma}_2(F(b))=\bigcup\limits_{y\in \Gamma_2(F(b))}\{\{y\}\}.
$$
Since there is the one-to-one correspondence between the set of leaves of representing tree and the set of points of the corresponding ultrametric space, by~(\ref{eq120}) we have $f^* (\widetilde{\Gamma}_1(b)) = \widetilde{\Gamma}_2 (F(b))$ which imply~(\ref{eq110}).

The arguing for $f^{-1}$ is analogous.
\end{proof}

\begin{remark}
The proof of Theorem~\ref{t32} is a modification of the proof form~\cite{P(TIAMM)}. Another proof of this result was suggested by Anton Lunyov~\cite{LPD}.
\end{remark}

\begin{definition}\label{d3.2}
	Let $T_i=T_i(r_i,l_i)$ be labeled rooted trees with the roots $r_i$ and the labelings $l_i\colon V(T_i)\to L_i$, $i=1$, $2$ such that $L_1=L_2$. An isomorphism $f\colon V(T_1) \to V(T_2)$ of the rooted trees $T_1(r_1)$ and $T_2(r_2)$ is an isomorphism of the labeled rooted trees $T_1(r_1,l_1)$ and $T_2(r_2,l_2)$ if the equality \begin{equation}\label{d3.2e1}
		l_2(f(v))=l_1(v)
	\end{equation}
holds for every $v \in V(T_1)$. The labeled rooted trees $T_1(r_1,l_1)$ and $T_2(r_2,l_2)$ are isomorphic if there is an isomorphism of these trees.
\end{definition}

Recall that two metric spaces $(X,d)$ and $(Y, \rho)$ are isometric if there is a bijection $f\colon X\to Y$ such that the equality
$$
d(x,y)=\rho(f(x),f(y))
$$
holds for all $x$, $y \in X$.

The following result was formulated in~\cite[Theorem 2.6]{DPT(Howrigid)} with a sketch of the proof.

\begin{theorem}\label{l3.3}
	Let $(X,d)$ and $(Y, \rho)$ be nonempty finite ultrametric spaces. Then the labeled rooted trees $T_X$ and $T_Y$ with the labelings defined as in~\eqref{e2.7} are isomorphic if and only if $(X,d)$ and $(Y, \rho)$ are isometric.
\end{theorem}
\begin{proof}
	Let $\Psi\colon X \to Y$ be an isometry. It follows from Proposition~\ref{lbpm} and the definition of the representing trees that
	$$
	V(T_X)=\mathbf{B}_X \quad \text{and} \quad V(T_Y)=\mathbf{B}_Y.
	$$
	Since $\Psi\colon X \to Y$ is an isometry, the set
	$$
	\{\Psi(x)\colon x \in A\}, \quad A \subseteq X
	$$
	is a ball in $(Y, \rho)$ if and only if $A$ is a ball in $(X,d)$. Consequently we can define a bijection
	$f\colon \mathbf{B}_X \to \mathbf{B}_Y$ by the rule
	\begin{equation}\label{l3.3e1}
	f(B)=\{\Psi(x)\colon x \in B\}, \quad B \in \mathbf{B}_X.
	\end{equation}
	A ball $B_1 \in \mathbf{B}_X$ is a direct successor of a ball $B_2 \in \mathbf{B}_X$ if and only if $B_1\neq B_2$ and $B_1\subseteq B_2$ and
	\begin{equation}\label{l3.3e2}
	(B_1\subseteq B \subseteq B_2) \Rightarrow ((B_1=B) \vee (B_2=B))
	\end{equation}
holds for every $B \in \mathbf{B}_X$. Since $\Psi$ is bijective, we have $(\Psi(B_1) \neq \Psi(B_2))\Leftrightarrow(B_1\neq B_2)$ and $(\Psi(B_1) \subseteq \Psi(B_2))\Leftrightarrow(B_1\subseteq B_2)$ and \eqref{l3.3e2} is an equivalent of
	$$
	(\Psi(B_1)\subseteq \Psi(B) \subseteq \Psi(B_2)) \Rightarrow ((\Psi(B_1)=\Psi(B)) \vee (\Psi(B_2)=\Psi(B))).
	$$
	Thus~\eqref{d3.1e1} holds for all $B_1$, $B_2 \in \mathbf{B}_X$ if $f$ is defined by~\eqref{l3.3e1} and $T_1=T_X$ and $T_2=T_Y$. Moreover, for $T_1=T_X$ and $T_2=T_Y$, equality~\eqref{d3.2e1} is an equivalent for
	$$
	\operatorname{diam}(\Psi(B))=\operatorname{diam}(B),
	$$
	that holds because $\Psi$ is an isometry. Moreover it is clear that $f(X)=Y$. Hence the labeled rooted trees are isomorphic if $(X,d)$ and $(Y, \rho)$ are isometric.
	
	The converse is also valid. To see it, note that the restriction of an isomorphism
	$$
	f\colon V(T_X) \to V(T_Y)
	$$
	of labeled rooted trees $T_X$ and $T_Y$ on the set $\overline{L}_T$ of leaves of $T_X$ gives us the bijection
	\begin{equation}\label{l3.3e3}
	\overline{L}_{T_X} \ni v \mapsto f(v) \in \overline{L}_{T_Y}.
	\end{equation}
	Since $v$ and $f(v)$ are some one-point subsets of $X$ and of $Y$ respectively, there is the unique bijection $\Psi\colon X\to Y$ such that
	$$
	(\Psi(x)=y) \Leftrightarrow (f(\{x\})=\{y\})
	$$
	holds for all $x \in X$ and $y \in Y$. Lemma~\ref{l2} implies that $\Psi$ is an isometry of $(X,d)$ and $(Y, \rho)$.
\end{proof}

In the rest of the section we will consider a class of mappings (of finite ultrametric spaces) which is wider than the class of isometries but narrower than the class of ball-preserving mappings.
\begin{definition}\label{d3.1.1}
Let $(X_1, d_1)$ and $(X_2, d_2)$
be finite ultrametric spaces. A bijection $\Phi \colon X_1\to X_2$ is a weak similarity if there is a strictly increasing bijective function $f\colon \Sp{X_1} \to \Sp{X_2}$ such that the equality
\begin{equation}\label{ex1}
d_1(x,y)=f(d_2(\Phi(x), \Phi(y)))
\end{equation}
holds for all $x,y \in X_1$.
We say that $(X_1,d_1)$ and $(X_2,d_2)$ are weakly similar if there is a weak similarity $\Phi\colon X_1\to X_2$.
\end{definition}
\begin{remark}\label{r3.1.2}
The weak similarities of arbitrary semimetric spaces were studied in~\cite{DP}. See also ~\cite{KL} and references therein for some results related to weak similarities of subsets of Euclidean finite-dimensional spaces.
\end{remark}

\begin{definition}\label{d3.1.2}
Let $T_i=T_i(r_i,l_i)$ be finite labeled rooted trees with the roots $r_i$ and the surjective labelings $l_i\colon V(T_i)\to L_i$, $L_i\subseteq \mathbb R^+$,  $i=1,2$. An isomorphism $f\colon V(T_1)\to V(T_2)$ of rooted trees $T_1(r_1)$ and $T_2(r_2)$ is a weak isomorphism of labeled rooted trees $T_1(r_1,l_1)$ and $T_2(r_2,l_2)$ if there is a strictly increasing bijection $\Psi\colon L_2 \to L_1$ such that the diagram
\begin{equation}\label{ex2}
\begin{diagram} \node{V({T_1})} \arrow[2]{e,t}{f} \arrow{s,l}{l_1} \node[2]{V({T_2})} \arrow{s,r}{l_2}  \\
 \node{L_1}  \node[2]{L_2}\arrow[2]{w,b}{\Psi} \end{diagram}
 \end{equation}
is commutative.
We say that $T_1(r_1,l_1)$ and $T_2(r_2,l_2)$ are weakly isomorphic if there is a weak isomorphism $f\colon V(T_1)\to V(T_2)$.
\end{definition}

It is clear that every two weakly isomorphic $T_1(r_1,l_1)$ and $T_2(r_2, l_2)$ are isomorphic in the sense of Definition~\ref{d3.1}. Moreover, if $L_1=L_2$ and $T_1(r_1,l_1)$ and $T_2(r_2, l_2)$ are isomorphic in the sense of Definition~\ref{d3.2}, then $T_1(r_1,l_1)$ and $T_2(r_2, l_2)$  are weakly isomorphic with $\Psi(x)=x$ for every $x\in L_2$.

\begin{theorem}\label{t4.3*}
Let $(X_1, d_{1})$, $(X_2, d_{2})$ be nonempty finite ultrametric spaces with the spectra $\Sp{X_1}$, $\Sp{X_2}$ and the representing rooted trees $T_1=T_1(r_1, l_1)$, $T_2=T_2(r_2, l_2)$. The following conditions are equivalent.
\begin{itemize}
  \item [$(i)$] The ultrametric spaces $(X_1, d_{1})$ and $(X_2, d_{2})$ are weakly similar.
  \item [$(ii)$] The labeled rooted trees $T_1(r_1,l_1)$ and $T_2(r_2,l_2)$ are weakly isomorphic.
\end{itemize}
\end{theorem}
\begin{proof}
It is easy to see that the implication $(i)\Rightarrow(ii)$ is satisfied.
Suppose that there exist an isomorphism $f\colon V(T_1)\to V(T_2)$ of $T_1$ and $T_2$ and a strictly increasing bijection $\Psi \colon \Sp{X_2}\to \Sp{X_1}$ such that diagram~(\ref{ex2}) is commutative with $L_1 = \Sp{X_1}$ and $L_2=\Sp{X_2}$.
 To prove that $(ii)\Rightarrow(i)$ holds it suffices to show that the function $\hat{f}\colon X_1\to X_2$ defined by
$$
(\hat{f}(t)=z) \Leftrightarrow
(f(\{t\})= \{z\}),
$$
 is a weak similarity such that
\begin{equation}\label{e4.4*}
d_1(x,y)=\Psi(d_2(\hat{f}(x), \hat{f}(y)))
\end{equation}
holds for all $x,y \in X_1$.

For $x_1, y_1 \in X_1$, write $r_1=d_1(x_1,y_1)$. Let $B_1=\{x\in X_1\colon d_1(x,x_1)\leqslant r_1\}$. Then $x_1, y_1\in B_1\in V(T_1)$ and
\begin{equation}\label{e45*}
r_1=d_1(x_1,y_1)=\diam B_1=l_1(B_1).
\end{equation}
The ball $f(B_1)$ belongs to $V(T_2)$ and, using Lemma~\ref{l2}, we obtain that
\begin{equation}\label{e46*}
d_2(\hat{f}(x_1),\hat{f}(y_1))=\diam f(B_1) = l_2(f(B_1)).
\end{equation}

Since diagram~(\ref{ex2}) is commutative, the equality
$$
l_1(B_1)= \Psi(l_2(f(B_1)))
$$
follows. The last equality, ~(\ref{e45*}) and~(\ref{e46*}) imply~(\ref{e4.4*}). The function $\Psi\colon \Sp{X_1}\to \Sp{X_2}$ is a strictly increasing bijection. Thus $\hat{f}$  is a weak similarity of $(X_1, d_1)$ and $(X_2, d_2)$. The proof is completed.
\end{proof}

If $\Sp{X_1}=\Sp{X_2}$ holds for finite ultrametric spaces $(X_1,d_1)$ and $(X_2, d_2)$ which are weakly similar, then it is easy to prove that $(X_1, d_1)$ and $(X_2, d_2)$ are isomorphic. (See, for example, ~\cite{DP}, Corollary 3.7.) Hence Theorem~\ref{t4.3*} and Lemma~\ref{l3.3} imply the following corollary.

\begin{corollary}\label{c4.3}
Let $(X_1, d_1)$ and $(X_2, d_2)$ be nonempty finite ultrametric spaces with $$
\Sp{X_1}=\Sp{X_2}
$$
and let $T_1=T_1(r_1, l_1)$ and $T_2=T_2(r_2, l_2)$ be the labeled representing trees of $(X_1, d_1)$ and $(X_2, d_2)$ with labelings $l_1\colon V(T_1)\to \Sp{X_1}$ and $l_2\colon V(T_2)\to \Sp{X_2}$ defined as in~(\ref{e2.7}). Then $T_1(r_1, l_1)$ and $T_2(r_2,l_2)$ are weakly isomorphic in the sense of Definition~\ref{d3.1.2} if and only if $T_1(r_1,l_1)$ and $T_2(r_2,l_2)$ are isomorphic in the sense of Definition~\ref{d3.2}.
\end{corollary}

The concept of weak similarity of metric spaces is closely connected to such concepts as ordinal scaling, non-metric multidimensional scaling, ranking or isotonic embedding which were used in the psychometric studies and the machine learning~\cite{AWCLKB, BG, JN,K64a, K64b, LGCL, QY, RF, S62b, S66, TLBSK, WJJ}. The close concept of generalized ultrametric spaces was introduced by S. Priess-Crampe and P. Ribenboim in~\cite{PC, PCR}. These spaces were studied in ~\cite{PCR1,PCR2,R96,R09} and lead to interesting applications in logic programming, computational logic and domain theory~\cite{SH,K,PCR1}.

M. Kleindesner and U. von Luxburg write~\cite{KL}: "Even though ordinal embeddings dates back to the 1960ies and is widely used in practice, surrisingly little is known about its theoretical properties." In this connection, it seems that the future studies of the interrelations between metric spaces, generalized ultrametric spaces and graphs are actual and can lead to adequate mathematical models for important problems of machine learning, data analysis and psychometrics.

\section{Embeddings of rooted trees in balleans of ultrametric spaces}

Recall that we write $T_1\simeq T_2$ if $T_1= T_1(r_1)$ and $T_2=T_2(r_2)$ are isomorphic rooted trees and, for finite ultrametric space $(X,d)$,  denote by $\widetilde{T}_X$ the rooted subtree of the representing tree $T_X$ induced by the set $I(T_X)$ of all internal nodes of $T_X$.

\begin{lemma}\label{l3.4}
	Let $T=T(r,l_T)$ be a finite labeled rooted tree with the root $r$ and the surjective labeling $l_T\colon V(T)\to L$, with $L\subseteq \mathbb R^+$ and $0\in L$. Then the following two conditions are equivalent.
	\begin{enumerate}
		\item [$(i)$] For every $u \in V(T)$ we have $\delta^+(u)\neq 1$ and
		$$
		(\delta^+(u) =0) \Leftrightarrow (l_T(u)=0)
		$$
		and, in addition, the inequality
		\begin{equation}\label{l3.4e1}
		l_T(v) < l_T(u)
		\end{equation}
		holds whenever $v$ is a direct successor of $u$.
		\item [$(ii)$] There is a finite ultrametric space $(X,d)$ such that the equality $\Sp{X}=L$ holds and $T_X$, $T$ are isomorphic in the sense of Definition~\ref{d3.2}.
	\end{enumerate}
\end{lemma}
\begin{proof}
	$(i) \Rightarrow (ii)$. Let us denote by $X$ the set of the leaves of $T$. For every pair $x$, $y \in X$ denote by $P_{x,y}$ the subset of $V(T)$ consisting of all nodes $w$ for which $x$ and $y$ are successors of $w$ and write
	\begin{equation}\label{l3.4e2}
	d(x,y) = \min_{w \in P_{x,y}} l_T(w).
	\end{equation}
	Using condition $(i)$ we can prove that the function $d$ is an ultrametric on $X$. Now~\eqref{l3.4e2}, the definition of the representing trees and Lemma~\ref{l2} imply that $T_X$ and $T(r,l_T)$ are isomorphic in the sense of Definition~\ref{d3.2}.
	
	$(ii) \Rightarrow (i)$. If $(X,d)$ is a finite ultrametric space, then condition~$(i)$ evidently holds for $T=T_X$. Moreover, if we have two labeled rooted trees which are isomorphic in the sense of Definition~\ref{d3.2} and one of them satisfies condition~$(i)$, then another tree also satisfies~$(i)$. The implication $(ii) \Rightarrow (i)$ holds.
\end{proof}

Lemma~\ref{l3.4} implies, in particular, the following corollary.

\begin{corollary}\label{c3.5}
	Let $T=T(r)$ be a finite rooted tree. Then the following conditions are equivalent.
	\begin{enumerate}
		\item[$(i)$] For every $u \in V(T)$ we have $\delta^+(u)\neq 1$.
		\item[$(ii)$] There is a finite ultrametric space $(X,d)$ such that $T_X$ and $T$ are isomorphic in the sense of Definition~\ref{d3.1}.
	\end{enumerate}
\end{corollary}
\begin{remark}
The rooted trees described by Corollary~\ref{c3.5} have the following interpretation in evolutionary biology. The finite rooted tree $T$ with $|V(T)|\geqslant 2$ satisfies condition $(i)$ of Corollary~\ref{c3.5} if and only if $T$ is a rooted phylogenetic tree (see Definition 2.2.2 in~\cite{SS}).
\end{remark}

Using Corollary~\ref{c3.5} and Proposition~\ref{lbpm} we can easily show that, for every finite rooted tree $T(r)$, there is a finite ultrametric space $(X,d)$ such that $T(r)\simeq \widetilde{T}_X$. What is the minimal cardinality of finite ultrametric spaces $(X,d)$ satisfying $\widetilde{T}_X\simeq T(r)$ for given $T(r)$?

\begin{lemma}\label{l3.5*}
Let $T=T(r)$ be a rooted tree. For every $v\in V(T)$ denote by $\operatorname{DS}(v)$ the set of directed successors of $v$ which are leaves of $T$. Then the sets $\operatorname{DS}(v)$ and $\operatorname{DS}(u)$ are disjoint if $u$ and $v$ are distinct.
\end{lemma}
\begin{proof}
Let $u,v \in V(T)$ and $u\neq v$. Then there are a path $P_u$ joining $u$ and $r$ and a path $P_v$ joining $v$ and $r$. It is clear that
$$
V(P_u)\cap \operatorname{DS}(u)=V(P_v)\cap \operatorname{DS}(v)=\varnothing.
$$
Let $P\subseteq P_u\cup P_v$ be the path joining $u$ and $v$, $P=(x_1,...,x_n)$, $x_1=u$ and $x_n=v$. Suppose that $\operatorname{DS}(u)\cap \operatorname{DS}(v)\neq \varnothing$. If
$$
x_0 \in \operatorname{DS}(u)\cap \operatorname{DS}(v),
$$
then $(x_0,x_1,...,x_n)$ is a cycle in $T$, contrary to the definition of the trees.
\end{proof}

\begin{theorem}\label{t36}
Let $T=T(r)$ be a finite rooted tree. Then there is a finite ultrametric space $(X,d)$ such that $T \simeq \widetilde{T}_X$ and
\begin{equation}\label{e3.8}
|X|=\sum\limits_{v\in V(T)}(2-\delta^+(v))_{+},
\end{equation}

where
\begin{equation}\label{e3.9}
(2-\delta^+(v))_{+}=
\begin{cases}
2-\delta^+(v), &\text{ if } \ \, 2\geqslant \delta^+ (v), \\
0, &\text{ if } \ \, 2< \delta^+ (v).
\end{cases}
\end{equation}
Moreover, if $(Y,\rho)$ is a finite ultrametric space such that $T\simeq\widetilde{T}_Y$, then the inequality \begin{equation}\label{e3.10}
|Y|\geqslant\sum\limits_{v\in V(T)}(2-\delta^+(v))_{+},
\end{equation}
holds.
\end{theorem}
\begin{proof}
Let $\overline{L}_T$ be the set of leaves of $T$ and let
$$
V^1(T)=\{v \in V(T)\colon \delta^+ (v) =1\}.
$$
Let us consider the sets $W_0^1$, $W_0^2$ and $W_1$ such that $V(T)$, $W_0^1$, $W_0^2$ and $W_1$ are disjoint and $|W_0^1|=|W_0^2|=|\overline{L}_T|$ and $|W_1|=|V^1(T)|$. Let $f^1_0\colon \overline{L}_T\to W_0^1$, $f^2_0\colon \overline{L}_T\to W_0^2$ and $f_1\colon V_1(T)\to W_1$ be bijections. We define a graph $G$ as
$$
V(G)=V(T) \cup W_0\cup W_1^1 \cup W_1^2
$$
and $u,v \in V(G)$ are adjacent in $G$ if and only if $\{u,v\} \in E(T)$ or
$$
\{u,v\}=\{u, f^1_0(u)\} \ \,
\text{ and } \ \, u\in L_T,
$$
or
$$
\{u,v\}= \{u, f^2_0(u)\} \ \, \text{ and } \, \ u \in L_T,
$$
or
$$
\{u,v\}= \{u, f_1(u)\} \ \, \text{ and } \, \ u \in V^1(T).
$$
It is easy to prove that $G$ is a tree (see Figure~\ref{fig2*}). Suppose that the root $r$ of $T(r)$ is also the root of $G$. The set of leaves of $G(r)$ is $W_0^1\cup W_0^2 \cup W_1$, i.e.,
\begin{equation}\label{e3.11}
\overline{L}_G=W_0^1\cup W_0^2\cup W_1.
\end{equation}
For every vertex of the rooted tree $G(r)$ we have the inequality $\delta^+(x)\geqslant 2$. Hence by Lemma~\ref{l3.4} there is  a finite ultrametric space $(X,d)$ such that $G(r)\simeq T_X$ and $|X|=|\overline{L}_G|$. To finish the proof of the first statement of the theorem it suffices to note that the last equality,~(\ref{e3.11}) and~(\ref{e3.9}) imply~(\ref{e3.8}).

Suppose now that $(Y,\rho)$ is a finite ultrametric space such that $T(r)\simeq \widetilde{T}_Y$.
By definition of $\widetilde{T}_Y$ the equality $I(T_Y)=V(\widetilde{T}_Y)$ holds.  Consequently if there is $v\in V(\widetilde{T}_Y)$ such that $\delta^+(v)=1$ in $\widetilde{T}_Y$, then using Lemma~\ref{l3.4} we can find at least one direct successor $u$ of $v$ (in $T_Y$) such that $u\in \overline{L}_{T_Y}$. Similarly if $\delta^+(v)=0$ in $\widetilde{T}_Y$, then there exist at least two leaves $u_1$ and $u_2$ in $T_Y$ such that $u_1$ and $u_2$ are some distinct direct successors of $v$ in $T_Y$. Using Lemma~\ref{l3.5*} we obtain
 $$
 |\overline{L}(T_Y)|\geqslant \sum\limits_{v\in V(\widetilde{T}_Y)}(2-\delta^+(v))_+.
 $$
Since $\widetilde{T}_Y\simeq T$, the equality
$$
\sum\limits_{v\in V(\widetilde{T}_Y)}(2-\delta^+(v))_+ = \sum\limits_{v\in V(T)}(2-\delta^+(v))_+
$$
holds. Inequality~(\ref{e3.10}) follows.
\end{proof}
\begin{figure}
	\begin{tikzpicture}[scale=0.7]
\draw %(3,10) --
(3,7) -- (3,4) -- (5,2);
\draw (3,4) -- (1,2);
	\node [left=1cm] at (3,7) {$T(r)$};
%	\fill[black] (3,10) circle(3pt);
	\fill[black] (3,7) circle(3pt);
	\fill[black] (3,4) circle(3pt);
	\fill[black] (5,2) circle(3pt);
	\fill[black] (1,2) circle(3pt);

\draw %(10,10) --
(10,7) -- (10,4) -- (12,2) -- (13,0);
\draw (10,4) -- (8,2);
\draw (8,2) -- (7,0);
\draw (8,2) -- (9,0);
%\draw (10,10) -- (9,8);
\draw (10,7) -- (9,5);
\draw (12,2) -- (11,0);
	%\fill[black] (10,10) circle(3pt);
	\fill[black] (10,7) circle(3pt);
	\fill[black] (10,4) circle(3pt);
	\fill[black] (12,2) circle(3pt);
	\fill[black] (13,0) circle(3pt);

	\fill[black] (8,2) circle(3pt);
	\fill[black] (7,0) circle(3pt);
	\fill[black] (9,0) circle(3pt);
%	\fill[black] (9,8) circle(3pt);
	\fill[black] (9,5) circle(3pt);
    \fill[black] (11,0) circle(3pt);

	\node [right=1cm] at (10,7) {$G(r)$};

	\end{tikzpicture}
	\caption{Each tree is a tree of nonsingular balls of finite ultrametric space.}
	\label{fig2*}
\end{figure}
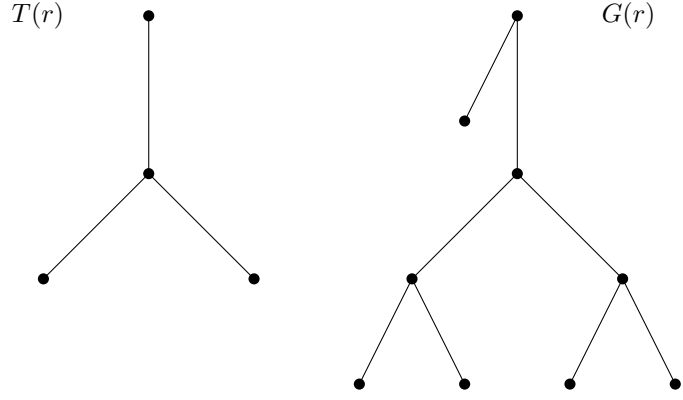

Analyzing the proof of Theorem~\ref{t36} we obtain the following corollaries.
\begin{corollary}\label{c3.8*}
Let $T=T(r)$  be a finite rooted tree and let $(X,d)$ and $(Y,\rho)$ be finite ultrametric spaces for which $T\simeq \widetilde{T}_X\simeq \widetilde{T}_Y$  and
$$
|X|=|Y|=\sum\limits_{v\in V(T)}(2-\delta^+(v))_{+}.
$$
Then $T_X\simeq T_Y$ holds.
\end{corollary}

\begin{corollary}\label{c4.6*}
Let $T=T(r)$  be a finite rooted tree and let $(Y,\rho)$ be a finite ultrametric space such that $\widetilde{T}_Y\simeq T$.
Then there is a subspace $X$ of $Y$ such that $\widetilde{T}_X\simeq T$ and
$$
|X|=\sum\limits_{v\in V(T)}(2-\delta^+(v))_{+}.
$$
 hold.
\end{corollary}

Let  $(X,d)$ be a metric space. Recall that the \emph{sphere} with a radius $t>0$ and a center $c\in X$ is the set
$$
S_t(c) = \{x\in X\colon d(x,c)=t\}.
$$

\begin{lemma}\label{l3.6.2}
Let $(X,d)$ be a finite ultrametric space and let $T_X$ be the representing tree of $(X,d)$. The following conditions are equivalent.
\begin{itemize}
  \item [$(i)$] For every nonsingular ball $B\in \mathbf B_X$ there are a point $c\in B$ and a number $t>0$ such that \begin{equation}\label{e3.12*}
      B=S_t(c)\cup \{c\}.
      \end{equation}
  \item [$(ii)$] For every internal node $v\in V(T_X)$ there is a leaf $u$ of $T_X$ such that $u$ is a direct successor of $v$ in $T_X$.
\end{itemize}
\end{lemma}
\begin{proof}
The equivalence $(i)\Leftrightarrow(ii)$ follows from Definition~\ref{d2}, Theorem~\ref{t13} and Proposition~\ref{lbpm}.
\end{proof}
\begin{remark}\label{r23.6.3}
Equality~(\ref{e3.12*}) means that the ball $B$ contains a center $c$ of $B$ such that each point of $B$ expect the point $c$ belongs to the sphere $S_t(c)$.
\end{remark}

Using Lemma~{\ref{l3.6.2}} and arguing as in the proof of Theorem~\ref{t36} we obtain the following result.
\begin{theorem}\label{t3.6.2}
Let $T=T(r)$ be a finite rooted tree. Then there is a finite ultrametric space $(X,d)$ such that:
\begin{itemize}
  \item [$(i)$] $T\simeq \widetilde{T}_X$;
  \item [$(ii)$] For every nonsingular ball $B\in \mathbf B_X$ there are $c\in B$ and $t>0$ such that~(\ref{e3.12*}) holds;
  \item [$(iii)$] We have the equality
      \begin{equation}\label{e3.13*}
      |X| = |V(T)|+|\overline{L}_T|,
      \end{equation}
      where $\overline{L}_T$ is the set of leaves of the rooted tree $T=T(r)$.
\end{itemize}
Moreover, if $(Y,\rho)$ is a finite ultrametric space satisfying $(i)$ and $(ii)$ with $X=Y$, then the inequality
$$
|Y|\geqslant |V(T)|+|\overline{L}_T|
$$
holds.
\end{theorem}

\begin{corollary}\label{c362}
Let $T=T(r)$ be a finite rooted tree and let $(X,d)$ and $(Y,\rho)$ be finite ultrametric spaces for which $T\simeq \widetilde{T}_X \simeq \widetilde{T}_Y$,
$|X|=|Y|=|V(T)|+|\overline{L}_T|$ and~(\ref{e3.12*}) holds for every $B\in \mathbf B_X$ and every $B\in \mathbf B_Y$.
Then we have $T_X \simeq T_Y$.
\end{corollary}
\begin{remark}
Corollary~\ref{c3.8*} and Corollary~\ref{c362} can be regarded as statements of uniqueness of ultrametric space $X$ (up to isomorphism of $T_X$).

It is easy to see that
$$
\sum\limits_{v\in V(T)}(2-\delta^+(v))_{+} \leqslant \sum\limits_{v\in V(T)}1+(1-\delta^+(v))_{+}=|V(T)|+|\overline{L}_T|
$$
holds for every finite rooted tree $T=T(r)$.
Moreover the equality
\begin{equation}\label{ez0}
\sum\limits_{v\in V(T)}(2-\delta^+(v))_{+} =|V(T)|+|\overline{L}_T|
\end{equation}
holds if and only if we have
\begin{equation}\label{ez1}
(2-\delta^+(v))_{+} =1+(1-\delta^+(v))_{+}
\end{equation}
for all $v\in V(T)$. The nonnegative integer solutions of equation~(\ref{ez1}) are $\delta^+(v)=1$ and $\delta^+(v)=0$, i.e., ~(\ref{ez1}) holds for a rooted tree $T=T(r)$ if and only if the tree $T$ is a path $P$ whose vertices can be numbered so that
$$
V(P)=\{x_0,...,x_k \}, \quad k=|V(T)|, \quad x_0=r,
$$
and
$$
E(P)=\{\{x_0,x_1\},..., \{x_{k-1}, x_k\}\}.
$$
\end{remark}

For every finite rooted tree $T$ denote by  $\mathbf X_T$ the class of all finite ultrametric spaces satisfying $\widetilde{T}_X\simeq T$.
Now using Theorem~\ref{t36} and Theorem~\ref{t3.6.2} we obtain the following result.
\begin{theorem}\label{t3.6.3*}
Let $T=T(r)$ be a finite rooted tree. Then the following conditions are equivalent.
\begin{itemize}
  \item [$(i)$] For every $(X,d) \in \mathbf X_T$ and every nonsingular ball $B\in \mathbf B_X$ there are $c\in B$ and $t>0$ such that $B=S_t(c)\cup \{c\}$.
  \item [$(ii)$] $T_X$ has exactly one internal node at each level except the last level for every $X\in \mathbf X_T$.
  \item [$(iii)$] The rooted tree $T=T(r)$ is a path $(x_0,...,x_k)$ with $x_0=r$ and $k=|V(T)|-1$.
\end{itemize}
\end{theorem}
\begin{remark}\label{r8}
Some extremal properties of finite ultrametric spaces $X$ for which $\tilde{T}_X$ is a path will be given in Theorem~\ref{t3.10} of the next section of the paper (see also Theorem 3.4 from~\cite{DPT(Howrigid)}).
\end{remark}

In the rest of the section we obtain an estimation of the minimal cardinality $|X|$ of finite ultrametric spaces $X \in \mathbf X_T$ which have a hamiltonian diametrical graphs $G_X^d$ (see Definition~\ref{d2}).

Recall that the \emph{degree} $\delta_G(v)$ of a vertex $v$ of a simple graph $G$ is the number of vertices $u\in V(G)$ which are adjacent with $v$. A graph $G$ is \emph{hamiltonian} if $G$ contains a Hamilton cycle.

The following lemma is the well-known Dirak theorem~\cite[Theorem 10.1.1]{Di}.
\begin{lemma}\label{l4.13}
Let $G$ be a simple graph with $|V(G)|\geqslant 3$. If the inequality
\begin{equation}\label{e4.11}
  \delta_G(v)\geqslant \frac{1}{2}|V(G)|
\end{equation}
holds for every $v\in V(G)$, then $G$ is hamiltonian.
\end{lemma}

\begin{theorem}\label{t4.14}
Let $(X,d)$ be a finite ultrametric space with $|X|\geqslant 3$ and let $(G_X, w)$ be a complete weighted graph such that $V(G_X)=X$ and, for all $x,y \in X$,
$$
w(\{x,y\})=d(x,y) \ \, \text { if } \ \, x\neq y.
$$
Then the following conditions are equivalent.
\begin{itemize}
  \item [$(i)$] There is a Hamilton cycle $C$ in $(G_X,w)$ such that $w(e_1)=w(e_2)$ for all $e_1, e_2 \in E(C)$.
  \item [$(ii)$] The diametrical graph $G_X^d=G_X^d[X_1,...,X_k]$ is hamiltonian.
  \item [$(iii)$] The inequality
  \begin{equation}\label{e4.12}
      2|X_j|\leqslant \sum\limits_{i=1}^k|X_i|
  \end{equation}
      holds for every part $X_j$ of the diametrical graph $G_X^d[X_1,...,X_k]$.
\end{itemize}
\end{theorem}
\begin{proof}
$(iii)\Rightarrow(i)$. Suppose inequality~(\ref{e4.12}) holds for every part $X_j$ of the diametrical graph $G_X^d=G_X^d[X_1,...,X_k]$. Let $x\in V(G_X^d)$. Then there is $j \in \{1,...,k\}$ such that $x\in X_j$. Since $G_X^d$ is complete $k$-partite, the equality
\begin{equation}\label{e4.13}
     \delta_{G_X^d}(x)= \sum\limits_{ \substack{i=1 \\ i \neq j} }^k|X_i|
  \end{equation}
holds. Inequality~(\ref{e4.12}) implies
\begin{equation}\label{e4.14}
      |X_j|\leqslant \frac{1}{2}\sum\limits_{i=1}^k|X_i|.
  \end{equation}
It follows from~(\ref{e4.13}) and~(\ref{e4.14}) that
\begin{multline*}
\delta_{G_X^d}(x)= \sum\limits_{i=1}^k|X_i|-|X_j|=|V(G_X^d)|
\geqslant |V(G_X^d)|-\frac12 \sum\limits_{i=1}^k|X_j|\\=|V(G_X^d)|-\frac12|V(G_X^d)|
=\frac12|V(G_X^d)|.
\end{multline*}
Since $x$ is an arbitrary vertex of $G_X^d$ and $|V(G_X^d)|=|X|\geqslant 3$, the diametrical graph $G_X^d$ is hamiltonian by Lemma~\ref{l4.13}.

$(ii)\Rightarrow(i)$. This implication is trivial.

$(i)\Rightarrow(ii)$. Let $(i)$ hold. Then there are $t\in \Sp{X}$ and  a Hamilton cycle $C^*=(x_1,...,x_n)$ in $(G_X,w)$ such that
\begin{equation}\label{e4.15}
 d(x_i, x_{i+1})=w(\{x_i, x_{i+1}\})=t
\end{equation}
for every $i\in \{1,...,n-1\}$. Let us consider the balls
$$
B_t(x_i)=\{x\in X\colon d(x,x_i)\leqslant t\}\in \mathbf B_X, \ \, i=1,...,n.
$$
Since $(X,d)$ is ultrametric,~(\ref{e4.15}) implies that $t=\diam B_t(x_i)$ for every $i\in \{1,...,n\}$. Now the statements
$$
B_t(x_i)\cap B_t(x_{i+1})\neq \varnothing \ \, \text{ and } \ \, \diam B_t(x_i)=\diam B_t (x_{i+1})
$$
and the ultrametricity of $(X,d)$ imply that $B_t(x_i)=B_t(x_{i+1})$ holds for every $i \in \{1,...,n-1\}$. Thus $B_t(x_1)=B_t(x_2)=...=B_t(x_n)$. Since $C^*$ is a Hamilton cycle, we have
$$
X = \bigcup\limits_{i=1}^n B_t(x )= B_t(x_1).
$$
The equality $X=B_t(x_1)$ holds if and only if $t=\diam X$. It means that $C^*$ is a Hamilton cycle in $G_X^d=G_X^d[X_1,...,X_k]$.

Let $X_{j_0}$ be a part of $G_X^d[X_1,...,X_k]$, $j_0 \in \{1,...,k\}$ and let $u,z$ be distinct points in $X_{j_0}$. Since $C^*$ is a Hamilton cycle in $C_X^d[X_1,...,X_k]$ there are distinct $e^1_z, e^2_z\in E(C^*)$ and distinct $e^1_y, e^2_y \in E(C^*)$ such that $z\in e^1_Z\cap e^2_Z$ and $y\in e^1_y\cap e^2_y$ and

$$
\{e_{z,C^*}^1, e_{z,C^*}^2\}\cap \{e_{y,C^*}^1, e_{y,C^*}^2\}=\varnothing.
$$
Consequently the inequality $2|X_{j_0}|\leqslant E(C^*)$ holds. For every cycle $C$ we have $|V(C)|=|E(C)|$. Hence we have
$$
2|X_{j_0}|\leqslant |V(C^*)|=|X|=\sum\limits_{i=1}^k|X_j|,
$$
for every $j_0\in \{1,...,k\}$.
Inequality~(\ref{e4.12}) follows.
\end{proof}

\begin{theorem}\label{t4.16}
Let $T=T(r)$ be a finite rooted tree with the root $r$, let $v_1,...,v_k$, $k\geqslant 1$,  be the direct successors of $r$ in $T$ and let $T_{v_1},...,T_{v_k}$ be the induced rooted subtrees of $T$ defined by~(\ref{e2.1}). Then there is a finite ultrametric space $(X,d)$ such that
\begin{itemize}
  \item [$(i)$] $T\simeq \widetilde{T}_X$;
  \item [$(ii)$] The diametrical graph $G_X^d$ is hamiltonian;
  \item [$(iii)$] The equality
  \begin{equation}\label{e4.16}
    |X|=
    \begin{cases}
    2\sum\limits_{u\in V(T{v_1})}(2-\delta^+(u))_+, &\text{if } \, \ k=1,\\
    \sum\limits_{u\in V(T)}(2-\delta^+(u))_+ \\ + \left(2 \max\limits_{1\leqslant i \leqslant k}\sum\limits_{u\in V(T{v_i})}(2-\delta^+(u))_+
      -\sum\limits_{u\in V(T)}(2-\delta^+(u))_+\right)_+,
     &\text{if } \, \ k\geqslant 2,
    \end{cases}
  \end{equation}
\end{itemize}
holds. Moreover if $(Y,\rho)$ is a finite ultrametric space satisfying $(i)$ and $(ii)$, then the inequality
 \begin{equation}\label{e4.16*}
    |Y|\geqslant
    \begin{cases}
    2\sum\limits_{u\in V(T{v_1})}(2-\delta^+(u))_+, &\text{if } \, \ k=1,\\
    \sum\limits_{u\in V(T)}(2-\delta^+(u))_+ \\ + \left(2 \max\limits_{1\leqslant i \leqslant k}\sum\limits_{u\in V(T{v_i})}(2-\delta^+(u))_+
      - \sum\limits_{u\in V(T)}(2-\delta^+(u))_+\right)_+,
     &\text{if } \, \ k\geqslant 2,
    \end{cases}
  \end{equation}
holds.
\end{theorem}
\begin{proof}
Consider first the case $k=1$. By Theorem~\ref{t36} there is a finite ultrametric space $(X^1, d^1)$ such that $T_{v_1}\simeq\widetilde{T}_{X^1}$ and
$$
|X^1|=\sum\limits_{u\in V(T_{v_1})}(2-\delta^+(u))_+.
$$
Let $S$ be a positive real number such that $S>\diam X^1$ and let $Z^1$ be a set such that $|Z^1|=|X^1|$ and $Z^1\cap X^1=\varnothing$. Let us define on $X=X^1\cup Z^1$ the ultrametric $d$ by the rule
$$
d(x,y)=
\begin{cases}
d^1(x,y), &\text{if} \ \, x,y,\in X^1;\\
0, &\text{if} \ \, x=y;\\
S, &\text{otherwise}.
\end{cases}
$$
Then
$$
G_X^d=G_X^d[X^1,\{z_1\},...,\{z_m\}]
$$
where $\{z_1\},...,\{z_m\}$ are the one-point subsets of $Z^1$. It is clear that equality~(\ref{e4.16}) holds and, by Theorem~\ref{t4.14}, $G_X^d$ is hamiltonian.

Let $(Y, \rho)$ be a finite ultrametric space for which $T \simeq \widetilde{T}_Y$ and such that the diametrical graph $G_Y^d=G_Y^d[Y_1,...,Y_n]$ is hamiltonian. Since $\delta^+(r)=1$ (in $T$), the statement $T\simeq \widetilde{T}_Y$ implies that there is a unique $Y_j$, $j\in \{1,...,n\}$ with $|Y_j|\geqslant 2$. Without loss of generality we can set that $|Y_1|\geqslant 2$. Then for the subspace $Y_1$ of the ultrametric space $(Y,\rho)$ we have $\widetilde{T}_{Y_1}\simeq T_{v_1}$. By Theorem~\ref{t36} the inequality
\begin{equation}\label{e4.17}
  |Y_1|\geqslant \sum\limits_{u\in V(T_{v_1})}(2-\delta^+(u))_+
\end{equation}
holds. Moreover using Theorem~\ref{t4.14} we obtain the inequality
\begin{equation}\label{e4.18}
n-1=\sum\limits_{i=2}^n|Y_i|\geqslant |Y_1|.
\end{equation}
Now~(\ref{e4.16*}) follows from~(\ref{e4.17}), ~(\ref{e4.18}) and the equality
$$
|Y|=\sum\limits_{i=1}^n|Y_i|.
$$

For the case $k\geqslant 2$, the proof  is similar. Note that Corollary~\ref{c4.6*} can be used for verification of inequality~(\ref{e4.16*}).
\end{proof}

\section{From isomorphic rooted trees to isometric spaces}

It follows form Theorem~\ref{l3.3} that any two isometric finite ultrametric spaces have isomorphic representing trees. The converse does not hold. The spaces $(X,d)$ and $(Y, \rho)$ can fail to be isometric even if $T_X$ is isomorphic to $T_Y$ as rooted trees and $\operatorname{Sp}(X)=\operatorname{Sp}(Y)$.

Let us denote by $\mathcal{TSI}$ (tree-spectrum isometric) the class of all finite ultrametric spaces~$(X,d)$ which satisfy the following condition: If $(Y,\rho)$ is a finite ultrametric space such that $T_X \simeq T_Y$ and $\operatorname{Sp}(X) = \operatorname{Sp}(Y)$, then $(X,d)$ and $(Y,\rho)$ are isometric.

In this section we describe characteristic structural properties of representing trees $T_X$ for spaces $(X,d)$ belonging to some subclasses of~$\mathcal{TSI}$.

Let $T=T(r)$ be a rooted tree. The \emph{height} of $T$ is the number of edges on the longest path between the root and a leaf of $T$. The height of $T$ will be denoted by $h(T)$. For every node $v$ of $T$, the level of $v$ can be defined by the following inductive rule: The level of $r$ is zero and if $v \in V(T)$ has a level $x$, then every direct successor of $v$ has the level $x+1$. We denote by $\operatorname{lev}(v)$ the level of a node $v \in V(T)$. Thus
\begin{equation}\label{e3.1}
	h(T)= \max_{v \in \overline{L}_T} \operatorname{lev}(v).
\end{equation}

\begin{theorem}\label{t3.5}
	Let $(X,d)$ be a finite ultrametric space with $|X|\geqslant 2$. Suppose that the representing tree $T_X$ has an injective internal labeling. Then the following statements are equivalent.
	\begin{enumerate}
		\item [$(i)$] The space $(X,d)$ belongs to $\mathcal{TSI}$.
		\item [$(ii)$] The equality $\operatorname{lev}(v)=\operatorname{lev}(u)$ implies
		$$
		\operatorname{lev}(v)=\operatorname{lev}(u)=h(T_X)-1 \text{ and } \delta^+(u)=\delta^+(v)
		$$
		for all distinct internal nodes $u$, $v \in V(T_X)$.
	\end{enumerate}
\end{theorem}
\begin{proof}
	The root of $T_X$ is the unique node having the zero level and, in addition, each node $u$ with $\operatorname{lev}(u)=h(T_X)$ is a leaf of $T_X$. Hence if $(ii)$ does not hold, then we have exactly one from the following two alternatives:
	\begin{enumerate}
		\item [$(i_1)$] There are two distinct internal nodes $u$, $v \in V(T_X)$ and $k\in \{1, \ldots, h(T_X)-2\}$ such that
		\begin{equation*}%\label{t3.5e1}
		\operatorname{lev}(u)=\operatorname{lev}(v)=k;
		\end{equation*}
		\item [$(i_2)$] For every $k\in \{1, \ldots, h(T_X)-2\}$ there exists exactly one $w \in I(T_X)$ with $\operatorname{lev}(w)=k$ but there exist at least two distinct nodes $u$, $v \in I(T_X)$ such that
		\begin{equation}\label{t3.5e2}
		\operatorname{lev}(u)=\operatorname{lev}(v)=h(T)-1 \text{ and } \delta^+(u)\neq\delta^+(v).
		\end{equation}
	\end{enumerate}
$(i) \Rightarrow (ii)$. Let $(X,d) \in \mathcal{TSI}$. Suppose $(i_1)$ holds. Denote by $k_0$ the smallest $$
k\in \{1, \ldots, h(T_X)-2\}
$$
which satisfies $(i_1)$. Let $u^0, v^0 \in I(T_X)$ and $u^0\neq v^0$ and $\operatorname{lev}(u^0)=\operatorname{lev}(v^0)=k_0$.  If for every $w \in I(T_X)$ with $\operatorname{lev}(w)=k_0$ all children of $w$ are leaves of $T_X$, then we have
	$$
	k_0=h(T_X)-1.
	$$
	The last equality contradicts the condition $k_0\in \{1, \ldots, h(T_X)-2\}$. Hence there is $w \in I(T_X)$ satisfying the equality
	$$
	\operatorname{lev}(w)=k_0
	$$
	and having a direct successor which is an internal node of $T_X$. Without loss of generality we suppose $w\neq u^0$. Define the sets $W$ and $U$ as
	$$
	W=I(T_w) \text{ and } U=I(T_{u^0}),
	$$
	(see~\eqref{e2.1}). Since $\operatorname{lev}(u^0)=\operatorname{lev}(w)$, the sets $W$ and $U$ are disjoint. Write $n=|W|$ and $m=|U|$. Then we have $n+m\geqslant 3$. Let $l\colon V(T_X)\to \Sp{X}$ be the labeling of $T_X$ defined by~\eqref{e2.7} and let
	$$
	S_w = \{l(t)\colon t \in W\} \text{ and } S_{u^0} = \{l(t)\colon t \in U\}.
	$$
	Since the restriction $l|_{I(T_X)}$ is injective, we have
	$$
	|S_w|=n, \quad |S_{u^0}|=m \text{ and } S_w\cap S_{u^0}=\varnothing
	$$
	and, in addition,
	$$
	0\notin S_w \cup S_{u^0}.
	$$
	All elements $l^w \in S_w$ and $l^{u^0} \in S_{u^0}$ can be enumerated such that
	\begin{equation}\label{t3.5e3}
	l_1^w < l_2^w < \ldots < l_n^w \text{ and } l_1^{u^0} < l_2^{u^0} < \ldots < l_m^{u^0}.
	\end{equation}
	Since $n+m\geqslant 3$, there are disjoint subsets $S_w^*$ and $S_{u^0}^*$ of the set $S_w \cup S_{u^0}$ such that
	$$
	S_{u^0}^* \cup S_w^* = S_w \cup S_{u^0}
	$$
	and
	$$
	|S_w^*|=n, \quad |S_{u^0}^*|=m
	$$
	and
	\begin{equation}\label{t3.5e4}
	S_{u^0} \neq S_w^* \neq S_w \neq S_{u^0}^* \neq S_{u^0}.
	\end{equation}
	Similarly to~\eqref{t3.5e3} we suppose that the elements of $S_u^*$ and $S_w^*$ are enumerated as
	\begin{equation}\label{t3.5e5}
	l_1^{*w} < l_2^{*w} < \ldots < l_n^{*w} \text{ and } l_1^{*{u^0}} < l_2^{*{u^0}} < \ldots < l_m^{*{u^0}},
	\end{equation}
	$l_i^{*w} \in S_w^*$, $i=1,\ldots, n$, and $l_j^{*{u^0}}\in S_{u^0}^*$, $j=1,\ldots, m$.
	
	Let us define a new labeling $l^*$ on $V(T_X)$ by the rule
	\begin{equation}\label{t3.5e6}
	l^*(v)=\begin{cases}
	l(v), & \text{if } v \in V(T_X)\setminus (W\cup U),\\
	l_i^{*w}, & \text{if } v \in W \text{ and } l(v)=l_i^w,\ i\in \{1, \ldots, n\},\\
	l_j^{*{u^0}}, & \text{if } v \in U \text{ and } l(v)=l_j^{u^0},\ j\in \{1, \ldots, m\}.
	\end{cases}
	\end{equation}
	Using~\eqref{t3.5e3}, \eqref{t3.5e4} and Lemma~\ref{l3.4} we can prove that there is an ultrametric $d^*$ an the set $X^*=X$ such that the rooted trees $T_X$ and $T_{X^*}$ are isomorphic in the sense of Definition~\ref{d3.1} and
	$$
	\operatorname{Sp}(X)=\operatorname{Sp}(X^*).
	$$
	Since $(X,d)\in \mathcal{TSI}$, the trees $T_X$ and $T_{X^*}$ are also isomorphic in the sense of Definition~\ref{d3.2}. (See Theorem~\ref{l3.3}.) Let $f\colon V(T_X) \to V(T_{X^*})$ be the corresponding isomorphism. Since $f$ preserves the labeling, equality~\eqref{t3.5e6} imply that
	$$
	f(v)=v
	$$
	holds for every internal node $v \in V(T_X)\setminus (W\cup U)$. Moreover
	$$
	\operatorname{lev}(v)=\operatorname{lev}(f(v))
	$$
	holds for every $v \in V(T)$. Hence we have either $f(w)={u^0}$ or $f(w)=w$. If $f(w)={u^0}$, then $S_W=S_U^*$ holds, contrary to~\eqref{t3.5e4}. Similarly if $f(w)=w$, then we obtain $S_W=S_W^*$, that also contradicts~\eqref{t3.5e4}.
	
	Suppose $(i_2)$ is valid. Let $u$ and $v$ be two distinct internal nodes of $T_X$ for which~\eqref{t3.5e2} holds. Consider the new labeling $l^*\colon V(T_X)\to \Sp{X}$,
	$$
	l^*(w)=\begin{cases}
	l(w), & \text{if } u\neq w \neq v,\\
	l(u), & \text{if } w = v,\\
	l(v), & \text{if } w = u.\\
	\end{cases}
	$$
	By Lemma~\ref{l3.4} there is an ultrametric $d^*$ on the set $X^*=X$ such that $l^*$ is the labeling on the rooted tree $T_{X^*}$ generated by the ultrametric space $(X^*, d^*)$. It is easy to see that $T_X$ and $T_{X^*}$ are isomorphic as rooted trees and we have $\operatorname{Sp}(X^*) = \operatorname{Sp}(X)$. Now~$(i)$ implies that $(X,d)$ and $(X^*, d^*)$ are isometric, which is impossible because there are the unique ball $B$ in $(X,d)$ with the diameter $l(u)$ and the unique ball $B^*$ in $(X^*,d^*)$ with same diameter and $|B|=\delta^+(u)\neq \delta^+(v)=|B^*|$.
	
$(ii) \Rightarrow (i)$. Let $(Y, \rho)$ be a finite ultrametric space with $\operatorname{Sp}(Y) = \operatorname{Sp}(X)$ and let $f\colon V(T_X)\to V(T_Y)$ be an isomorphism of the rooted trees $T_X$ and $T_Y$. We first prove that the internal labeling of $T_Y$ is injective. This is equivalent to that
	\begin{equation}\label{t3.5e7}	|\mathbf{B}_Y|=|Y|+|\operatorname{Sp}(Y)|-1
	\end{equation}
	holds (see Theorem~\ref{t1}). The equality $\operatorname{Sp}(X)=\operatorname{Sp}(Y)$ implies
	\begin{equation}\label{t3.5e8}
	|\operatorname{Sp}(X)|=|\operatorname{Sp}(Y)|.
	\end{equation}
	Moreover, we have
	\begin{equation}\label{t3.5e9}
	|\mathbf{B}_Y|=|\mathbf{B}_X| \text{ and } |X|=|Y|
	\end{equation}
	because $f$ is an isomorphism of the rooted trees $T_X$ and $T_Y$. Since the internal labeling of $T_X$ is injective,
	\begin{equation}\label{t3.5e10}
	|\mathbf{B}_X|=|X|+|\operatorname{Sp}(X)|-1
	\end{equation}
	holds. Now \eqref{t3.5e7} follows from \eqref{t3.5e8}, \eqref{t3.5e9} and \eqref{t3.5e10}.
	
	Let $\widetilde{T}_X$ and $\widetilde{T}_Y$ be the subtrees of $T_X$ and, respectively, of $T_Y$ induced by $I(T_X)$ and, respectively, by $I(T_Y)$. Using $(ii)$ we can prove that the function $\widetilde{\Psi}\colon V(\widetilde{T}_X) \to V(\widetilde{T}_Y)$ defined by the rule
	\begin{equation}\label{t3.5e11}
	(\widetilde{\Psi}(u)=v) \Leftrightarrow (v \in \widetilde{T}_Y \text{ and } \operatorname{diam} u= \operatorname{diam} v)
	\end{equation}
is an isomorphism of the labeled rooted trees $\widetilde{T}_X$ and $\widetilde{T}_Y$. Note that $\widetilde{\Psi}$ is well-defined because the internal labelings of $T_X$ and $T_Y$ are injective and $\operatorname{Sp}(X) = \operatorname{Sp}(Y)$. The leaves of the trees $\widetilde{T}_X$ and $\widetilde{T}_Y$ are internal nodes of $T_X$ at the level $h(T_X)-1$ and, respectively, of $T_Y$ at the level $h(T_Y)-1$. By statement $(ii)$ the number of children is one and the same for all these nodes. Consequently $\widetilde{\Psi}$ can be extended to an isomorphism $\Psi \colon V(T_X) \to V(T_Y)$ of the labeled rooted trees $T_X$ and $T_Y$. By Theorem~\ref{l3.3}, $(X,d)$ and $(Y, \rho)$ are isometric.
\end{proof}

\begin{remark}
	Statement $(ii)$ means that $T_X$ has exactly one internal node at each level except the levels $h(T)$ and $h(T)-1$ and the number of children is constant for all internal nodes at the level $h(T)-1$. (See Figure~\ref{fig3} for example of a representing tree satisfying this property.)
\end{remark}

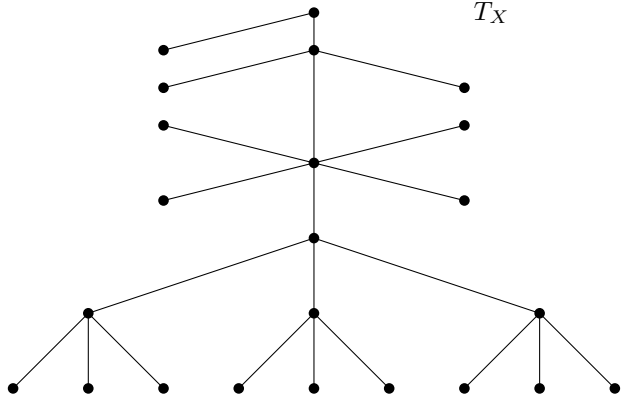
\begin{figure}
	\begin{tikzpicture}[scale=1]
	\draw (4,1) -- (4,6);
	\fill[black] (4,6) circle(2pt);
	\fill[black] (2,5.5) circle(2pt);
	\draw (4,6) -- (2,5.5);
	
	\fill[black] (4,5.5) circle(2pt);
	\fill[black] (2,5) circle(2pt);
	\fill[black] (6,5) circle(2pt);
	\draw (4,5.5) -- (2,5);
	\draw (4,5.5) -- (6,5);
	
	\fill[black] (4,4) circle(2pt);
	\fill[black] (2,3.5) circle(2pt);
	\fill[black] (6,3.5) circle(2pt);
	\fill[black] (2,4.5) circle(2pt);
	\fill[black] (6,4.5) circle(2pt);
	\draw (4,4) -- (2,3.5);
	\draw (4,4) -- (6,3.5);
	\draw (4,4) -- (2,4.5);
	\draw (4,4) -- (6,4.5);
		
	\fill[black] (4,3) circle(2pt);
	\fill[black] (1,2) circle(2pt);
	\fill[black] (7,2) circle(2pt);
	\draw (4,3) -- (1,2);
	\draw (4,3) -- (7,2);
	
	\fill[black] (4,2) circle(2pt);
	\fill[black] (3,1) circle(2pt);
	\fill[black] (4,1) circle(2pt);
	\fill[black] (5,1) circle(2pt);
	\draw (4,2) -- (3,1);
	\draw (4,2) -- (5,1);
	
	\fill[black] (0,1) circle(2pt);
	\fill[black] (1,1) circle(2pt);
	\fill[black] (2,1) circle(2pt);
	\draw (1,2) -- (0,1);
	\draw (1,2) -- (1,1);
	\draw (1,2) -- (2,1);
	
	\fill[black] (6,1) circle(2pt);
	\fill[black] (7,1) circle(2pt);
	\fill[black] (8,1) circle(2pt);
	\draw (7,2) -- (6,1);
	\draw (7,2) -- (7,1);
	\draw (7,2) -- (8,1);
\draw (6,6)  node [right] {$T_{X}$};

	\end{tikzpicture}
	\caption{The internal labeling of $T_X$ is injective if and only if $(X,d)\in \mathcal{TSI}$.}
	\label{fig3}
\end{figure}

\begin{theorem}\label{t3.7}
	Let $T = T(r)$ be a rooted tree with $\delta^+(u)\geqslant 2$ for every internal node $u$. Then the following conditions are equivalent.
	\begin{enumerate}
		\item[$(i)$] The tree $T$ contains exactly one internal node at the levels $1$, $\ldots$, $h(T)-2$ and at most two internal nodes at the level $h(T)-1$. Moreover, if $u$ and $v$ are different internal nodes with
$$		\operatorname{lev}(u)=\operatorname{lev}(v)=h(T)-1,
$$
then $\delta^+(u)=\delta^+(v)$ holds.
\item[$(ii)$] The statement $T_X \simeq T$ implies $(X,d) \in \mathcal{TSI}$ for every finite ultrametric space $(X,d)$.
	\end{enumerate}
\end{theorem}
\begin{proof}
	$(i)\Rightarrow (ii)$. Let condition (i) hold and let $X$, $Y$ be finite ultrametric spaces such that ${T}_X \simeq {T}\simeq {T}_Y$ and $\Sp{X}=\Sp{Y}$. Let $\Psi\colon V(T_X)\to V(T_Y)$ be an isomorphism of the rooted trees $T_X$ and $T_Y$. Suppose first there is exactly one internal node of $T$ at the level $h(T)-1$. Then it is easy to see that
	$$
		X\ni x \mapsto \{x\} \mapsto \Psi(\{x\}) \mapsto y \in Y,
	$$
	is an isometry, where $y$ is the unique point of the one-point set $\Psi(\{x\})$.
	
	Let $u$, $v \in I(T_X)$ and $\bar{u}$, $\bar{v} \in I(T_Y)$ and let
	$$
	\operatorname{lev}{u}=\operatorname{lev}{v}=\operatorname{lev}{\bar{u}}=\operatorname{lev}{\bar{v}}=h(T)-1
	$$
	and $u\neq v$ and $\bar{u} \neq \bar{v}$. There exist  two cases:
	\begin{itemize}
		\item [$(i_1)$] $|\Sp{X}|=|\Sp{Y}|=h(T)+1$;
		\item [$(i_2)$] $|\Sp{X}|=|\Sp{Y}|=h(T)+2$.
	\end{itemize}
Case ($i_1$) is trivial. Let ($i_2$) hold. Without loss of generality, suppose
$$
l(u)=l(\bar{u})\neq l(v)=l(\bar{v}).
$$
Let $\Psi\colon V(T_X)\to V(T_Y)$ be an isomorphism of $T_X$ and $T_Y$ such that $\Psi(u)=\bar{u}$ and $\Psi(v)=\bar{v}$. Doing direct calculations we see that
	$$
		X\ni x \mapsto \{x\} \mapsto \Psi(\{x\}) \mapsto y \in Y
	$$
	is an isometry between $X$ and $Y$, where $y$ is the unique element of the one-point set $\Psi(\{x\})$.
	
$(ii)\Rightarrow (i)$. Suppose the implication
	\begin{equation}\label{t3.7e1}
	(T_X \simeq T(r)) \Rightarrow ((X,d) \in \mathcal{TSI})
	\end{equation}
	holds for every finite ultrametric space $(X,d)$. Let us consider a labeling $l_T$ on $V(T)$ such that the restriction $l_T|_{I(T)}$ is injective and $l_T(u)=0$ for every $u \in \overline{L}_T$ and $l_T(v) < l_T(w)$, whenever $v$ is a direct successor of $w$. By Lemma~\ref{l3.4} there is a finite ultrametric space $(X,d)$ such that $T_X$ and the labeled rooted tree $T(r, l_T)$ are isomorphic. Theorem~\ref{t3.5} implies
	$$
	\operatorname{lev}(v)=\operatorname{lev}(u)=h(T_X)-1 \text{ and } \delta^+(u)=\delta^+(v)
	$$
	for all distinct $u$, $v \in I(T_X)$. Since $T_X$ and $T=T(r, l_T)$ are isomorphic, the similar statement holds for $T(r, l_T)$. Suppose now that
	$$
	u_1, \ldots, u_p \in I_T, \quad p\geqslant 3
	$$
	and
	$$
	\operatorname{lev}(u_1)= \ldots =\operatorname{lev}(u_p) = h(T)-1.
	$$
	Let $l^j\colon V(T)\to \mathbb R^+$, $j=1, 2$, be labelings of $T$ such that
	$$
	l^1(u)=l^2(u)=0
	$$
	for all $u \in \overline{L}_{T}$ and
	$$
	l^j(v) < l^j(w), \quad j=1,2
	$$
	whenever $v$ is a direct successor of $w$ and that
	\begin{equation}\label{t3.7e2}
	l^1(u_1)\neq l^2(u_1),\quad l^2(u_k)= l^1(u_1),\quad l^1(u_k)= l^2(u_1)
	\end{equation}
	for $k=2$, $\ldots$, $p$. Let $(X_1, d_1)$ and $(X_2, d_2)$ be finite ultrametric spaces,  corresponding respectively to the labeled rooted trees $T(r,l^1)$ and $T(r,l^2)$. Then $\operatorname{Sp}(X_1) = \operatorname{Sp}(X_2)$ but $(X_1, d_1)$ and $(X_2, d_2)$ are not isometric, contrary to~\eqref{t3.7e1}.
\end{proof}

Theorem~\ref{t3.7} gives us a sufficient condition, for $(X,d) \in \mathcal{TSI}$, which is independent of the properties of $\operatorname{Sp}(X)$. We may also provide the membership $(X,d) \in \mathcal{TSI}$ using only~$\operatorname{Sp}(X)$.

\begin{proposition}\label{p3.6}
	Let $A$ be a finite subset of $\mathbb{R}^+$ and let $0\in A$. Then the following conditions are equivalent
	\begin{enumerate}
		\item [$(i)$] The inequality $|A| \leqslant 3$ holds.
		\item [$(ii)$] For every finite ultrametric space $(X,d)$ the equality $\operatorname{Sp}(X) = A$ implies $(X,d) \in \mathcal{TSI}$.
	\end{enumerate}
\end{proposition}

The proof is simple, so that we omit it here. See Figure~\ref{fig4} for an example of representing tree of a finite ultrametric space $(X,d)$ with \mbox{$|\operatorname{Sp}(X)| = 3$}.

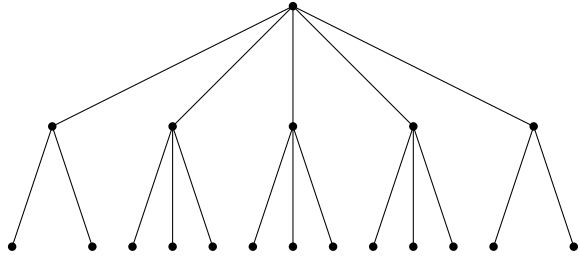
\begin{figure}
	\begin{tikzpicture}[scale=0.8]
	\draw (1,2) -- (5,4) -- (9,2);
	\draw (3,2) -- (5,4) -- (7,2);
	\draw (5,4) -- (5,2);
	\fill[black] (1,2) circle(2pt);
	\fill[black] (5,4) circle(2pt);
	\fill[black] (3,2) circle(2pt);
	\fill[black] (5,2) circle(2pt);
	\fill[black] (7,2) circle(2pt);
	\fill[black] (9,2) circle(2pt);
	
	\draw (1/3,0) -- (1,2) -- (5/3,0);
	\fill[black] (1/3,0) circle(2pt);
	\fill[black] (5/3,0) circle(2pt);
	
	\draw (7/3,0) -- (3,2) -- (11/3,0);
	\draw (3,2) -- (3,0);
	\fill[black] (7/3,0) circle(2pt);
	\fill[black] (3,0) circle(2pt);
	\fill[black] (11/3,0) circle(2pt);
	
	\draw (13/3,0) -- (5,2) -- (17/3,0);
	\draw (5,2) -- (5,0);
	\fill[black] (13/3,0) circle(2pt);
	\fill[black] (5,0) circle(2pt);
	\fill[black] (17/3,0) circle(2pt);
	
	\draw (19/3,0) -- (7,2) -- (23/3,0);
	\draw (7,2) -- (7,0);
	\fill[black] (19/3,0) circle(2pt);
	\fill[black] (7,0) circle(2pt);
	\fill[black] (23/3,0) circle(2pt);
	
	\draw (25/3,0) -- (9,2) -- (29/3,0);
	\fill[black] (25/3,0) circle(2pt);
	\fill[black] (29/3,0) circle(2pt);
	\end{tikzpicture}
	\caption{For every finite ultrametric space $(Y,\rho)$ with $h(T_Y)\leqslant 2$ there is a finite ultrametric space $(X,d)$ such that $|\Sp{X}|\leqslant 3$ and $T_Y\simeq T_X$.}
	\label{fig4}
\end{figure}

In the rest of the section we describe extremal properties of some $(X,d)\in \mathcal{TSI}$.

\begin{lemma}\label{p3.7}
	Let $(X,d)$ be a finite nonempty ultrametric space. Then the inequality
	\begin{equation}\label{p3.7e1}
	h(T_X) \leqslant |\operatorname{Sp}(X)|-1
	\end{equation}
	holds.
\end{lemma}
\begin{proof}
The inequality
	$$
	l(u) < l(v)
	$$
	holds whenever~$u$ is a direct successor of~$v$. Now~\eqref{p3.7e1} follows from~\eqref{e3.1}.
\end{proof}

\begin{lemma}\label{l3.8}
	The inequality
	\begin{equation}\label{l3.8e1}
	h(T_X)\leqslant |\mathbf{B}_X|-|X|
	\end{equation}
	holds for every finite nonempty ultrametric space $(X,d)$. This inequality becomes the equality if and only if
	$$
	h(T_X)=|\operatorname{Sp}(X)|-1\quad \text{and}\quad |\operatorname{Sp}(X)| = |\mathbf{B}_X|-|X|-1.
	$$
\end{lemma}
\begin{proof}
	Let $(X,d)$ be a finite nonempty ultrametric space. Using Lemma~\ref{p3.7} and Remark~\ref{r2.5} we obtain the double inequality
	$$
	h(T_X)+1\leqslant |\operatorname{Sp}(X)|\leqslant |\mathbf{B}_X|-|X|+1,
	$$
	that implies the desired result.
\end{proof}

\begin{lemma}\label{l3.9}
	Let $(X,d)$ be a finite ultrametric space with $|X|\geqslant 2$, and let $T_X$ have exactly one internal node at each level expect the last level. Then for every $Y\subseteq X$, $|Y|\geqslant 2$, the representing tree $T_Y$ of the space $(Y,d)$ also has exactly one internal node at each level expect the last level.
\end{lemma}
\begin{proof}
	Since every proper subset of $X$ can be obtained by consecutive deleting of points, it is sufficient to consider the case $Y=X\setminus \{x\}$, $x\in X$. Let $L_k$ be the set of leaves of $T_X$ at the level $k$ and let $n$ be a number of levels of $T_X$. If $x\in L_k$ and $|L_k|\geqslant 2$, then $T_Y$ can be obtained from $T_X$ by deleting the leaf $x$ from $T_X$. Clearly in this case $T_Y$ has exactly one internal node at each level expect the last level.
	
	Suppose $x \in L_k$ and $|L_k|=1$. Taking into consideration Proposition~\ref{lbpm}, the space $X$ can  be uniquely presented by a sequence of balls
	$$
	B_n\subset B_{n-1}\subset \cdots \subset B_2 \subset B_1
	$$
	where $B_n=L_n$ and $B_{i-1}=B_i\cup L_{i-1}$, $i=n,...,2$. Let us consider the space $(Y,d)$. It is clear that the following relations hold
	\begin{equation}\label{eq32}
		\overline{B}_n\subset \overline{B}_{n-1}\subset \cdots \subset \overline{B}_3 \subset \overline{B}_2
	\end{equation}
	where $$\overline{B}_n=B_n,...,\overline{B}_{k+1}=B_{k+1}, \quad \overline{B}_k=B_{k-1}\setminus\{x\},..., \overline{B}_2=B_1\setminus\{x\}.$$ Proposition~\ref{lbpm} and \eqref{eq32} imply that $T_Y$ fulfills the desired condition.
\end{proof}
%\newpage
\begin{theorem}\label{t3.10}
Let $(X,d)$ be a finite ultrametric space with $|X|\geqslant 2$. The following conditions are equivalent.
\begin{enumerate}
  \item[$(i)$] $T_X$ has exactly one internal node at each level except the last level.
  \item[$(ii)$] For every two distinct nonsingular balls $B_1$, $B_2 \in \mathbf{B}_X$ we have either $B_1 \subset B_2$ or $B_2 \subset B_1$.
  \item[$(iii)$] The equality
  \begin{equation}\label{t3.10e1}
  h(T_X)+|X|=|\mathbf{B}_X|
  \end{equation}
  holds.
  \item [$(iv)$] The equality
  $$
  h(T_Y)+|Y|=|\mathbf{B}_Y|
  $$
  holds for every nonempty $Y \subseteq X$.
\end{enumerate}
\end{theorem}
\begin{proof}
	The equivalence $\mathrm{(i)}\Leftrightarrow\mathrm{(ii)}$ follows from the definition of the representing trees. The equivalence $\mathrm{(i)}\Leftrightarrow\mathrm{(iii)}$ can be obtained from Lemma~\ref{l3.8}. The implication $\mathrm{(iv)}\Rightarrow \mathrm{(iii)}$ is trivial and $\mathrm{(iii)}\Rightarrow \mathrm{(iv)}$ follows from $\mathrm{(i)}\Leftrightarrow\mathrm{(ii)}$ and Lemma~\ref{l3.9}.
\end{proof}

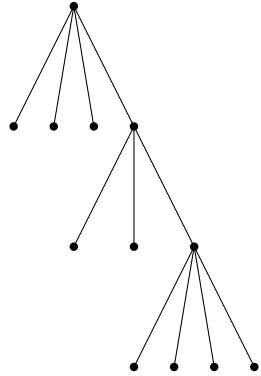
\begin{figure}
\begin{tikzpicture}[scale=0.8]
\draw (0,4) -- (1,6) -- (4,0);
\draw (2/3,4) -- (1,6);
\draw (4/3,4) -- (1,6);
\fill[black] (2/3,4) circle(2pt);
\fill[black] (4/3,4) circle(2pt);
\fill[black] (1,6) circle(2pt);
\fill[black] (0,4) circle(2pt);
\fill[black] (2,4) circle(2pt);
\fill[black] (3,2) circle(2pt);
\fill[black] (4,0) circle(2pt);
\draw (2,4) -- (1,2);
\draw (2,4) -- (2,2);
\draw (3,2) -- (2,0);
\draw (3,2) -- (8/3,0);
\draw (3,2) -- (10/3,0);
\fill[black] (2,2) circle(2pt);
\fill[black] (1,2) circle(2pt);
\fill[black] (2,0) circle(2pt);
\fill[black] (8/3,0) circle(2pt);
\fill[black] (10/3,0) circle(2pt);
\end{tikzpicture}
\caption{A tree which has exactly one internal node at each level except the last level.}
\label{fig6}
\end{figure}

Using condition~$\mathrm{(i)}$ from the last theorem it is easy to prove that every finite ultrametric space $(X, d)$ satisfying~\eqref{t3.10e1} belongs to $\mathcal{TSI}$ (see Figure~\ref{fig6}).

\textbf{Acknowledgements.} The research of the first author was supported by grant 0115U000136 of Ministry Education and Science of Ukraine.  The research of the second author was supported by National Academy of Sciences of Ukraine within scientific research works for young scientists,  Project 0117U006050, and by Project 0117U006353 from the Department of Targeted Training of Taras Shevchenko National University of Kyiv at the NAS of Ukraine. The first two authors were partially supported by the State Fund For Fundamental Research, Project F 71/20570.

\end{document}